\documentclass[11pt]{amsart}

\usepackage[pdfauthor={Morten S. Risager and Yiannis Petridis},
pdftitle={Dissolving cusp forms: Higher order Fermi's Golden Rules}
pdfsubjects={},
pdfkeywords={},
          pdfproducer={LaTeX2e with hyperref},
           pdfcreator={Pdflatex},
            pdfdisplaydoctitle=true
]{hyperref}
\usepackage{tikz} 
\usetikzlibrary{matrix,arrows}

\usepackage{xcolor}
\usepackage{amssymb}
\usepackage{amsmath}
\usepackage{latexsym}
\usepackage{amsthm}

\newtheorem{theorem}{Theorem}[section]
\newtheorem{lemma}[theorem]{Lemma}
\newtheorem{corollary}[theorem]{Corollary}

\theoremstyle{definition}

\newtheorem{example}[theorem]{Example}

\theoremstyle{remark}
\newtheorem{remark}[theorem]{Remark}

\newtheorem{proposition}[theorem]{Proposition}

\numberwithin{equation}{section}
\def \m {{l}} 
\def \g {{\gamma}}
\def \G {{\Gamma}}
\def \l {{\lambda}}
\def \a {{\alpha}}

\def \R {{\mathbb R}}
\def \H {{\mathbb H}}
\def \N {{\mathbb N}}
\def \C {{\mathbb C}}
\def \Z {{\mathbb Z}}

\def \e {{\epsilon }}

\def \GinfmodG {{\Gamma_{\!\!\infty}\!\!\setminus\!\Gamma}}
\def \GmodH {{\Gamma\backslash\H}}

\def \psl  {{\hbox{PSL}_2( {\mathbb R})} }
\def \L  {{\hbox{L}^2}}
\def \w {{\omega}}

\newcommand{\norm}[1]{\left\lVert #1 \right\rVert}
\newcommand{\abs}[1]{\left\lvert #1 \right\rvert}
\newcommand{\inprod}[2]{\left \langle #1,#2 \right\rangle}
\newcommand{\modsym}[2]{\left \langle #1,#2 \right\rangle}
\DeclareMathOperator*{\res}{res}

\begin{document}
\title[Higher order Fermi's Golden Rule]{Dissolving cusp forms: Higher order Fermi's Golden Rules}\author{Yiannis N. Petridis}
\address{Department of Mathematics, University College London, Gower Street, London WC1E 6BT, United Kingdom\\ Tel: +44 20-7679-7897, Fax: +44 20-7383-5519}\email{petridis@math.ucl.ac.uk}
\author{Morten S. Risager}\address{Department of Mathematical
  Sciences, University of Copenhagen, Universitetsparken 5, 2100 Copenhagen \O, Denmark } \email{risager@math.ku.dk}
 \subjclass{11F72\and  58J50}

\date{\today}
\maketitle

\begin{abstract}
For a hyperbolic surface embedded eigenvalues of the Laplace operator are unstable and tend to become resonances. A sufficient dissolving condition was identified by Phillips--Sarnak and is 
elegantly expressed in Fermi's Golden Rule. We prove formulas for higher approximations and obtain necessary and sufficient conditions for dissolving a cusp form with eigenfunction $u_j$
into a resonance. In the framework of perturbations in character varieties, we relate the result to the special values of  the $L$-series  $L(u_j\otimes F^n, s)$. This is the  
Rankin-Selberg convolution of $u_j$ with   $F(z)^n$, where $F(z)$ is
the antiderivative of a weight $2$ cusp form. In an example we show
that the  above-mentioned conditions  force  the
embedded eigenvalue to become a resonance in a punctured neighborhood of
the deformation space.

\end{abstract}
\section{Introduction}
For a hyperbolic surface with cusps the embedded eigenvalues of the Laplace operator $\Delta$ in the continuous spectrum 
are unstable. This is manifested by Fermi's Golden Rule developed in
\cite{phillips-sarnak2}. We describe the result in the simplest case
of a surface with one cusp and an eigenvalue of multiplicity one. Let $\lambda_j=1/4+r_j^2$ be an embedded eigenvalue with $r_j\in \R\setminus \{0\}$ and the corresponding $L^2$-normalized eigenfunction (Maa{\ss}  cusp form)
$u_j(z)$. Let $E(z, s)$ be the Eisenstein series, which on the critical line $\Re (s)=1/2$ is a generalized eigenfunction for $\Delta$, so that $E(z, 1/2+ir_j)$ corresponds in the same eigenvalue as the Maa{\ss} cusp form.  We set $s_j=1/2+ir_j$.
In \cite{phillips-sarnak1} Phillips and Sarnak identified a condition
that turns $\lambda_j$ into a resonance in Teichm\"uller space,
i.e. dissolving $\lambda_j$ into a resonance. In
\cite{sarnak1} Sarnak identified a similar condition for character varieties. 
Let $\Delta ^{(1)}$ denote the infinitesimal variation of the family
of Laplacians in either perturbation. Then the dissolving condition --
usually called the Phillips--Sarnak condition -- is
\begin{equation}\label{phillipssarnakcondition}\langle \Delta^{(1)} u_j, E(z, 1/2+ir_j)\rangle \ne 0.\end{equation}
In \cite{phillips-sarnak2} Phillips and Sarnak identified 
the dissolving condition in terms of the speed that the cuspidal
eigenvalue leaves the line $\Re (s)=1/2$ to become a resonance to the
left half-plane. If $s_j(\e)$ denotes the position of the resonance or
embedded cusp form, with perturbation series 
\begin{equation}
s_j(\e)=s_j+s_j^{(1)}(0)\e+\frac{s_j^{(2)}(0)}{2!}\e^2+\cdots ,\end{equation}
 then
\begin{equation}\label{fermi}\Re s_j^{(2)}(0)=-\frac{1}{4r_j^2}\left|\langle \Delta^{(1)} u_j, E(z, 1/2+ir_j)\rangle\right|^2.\end{equation}
Our aim in this paper is to investigate what happens when the
expression (\ref{fermi}) vanishes, or equivalently: what happens  if the Phillips-Sarnak
condition is \emph{not} satisfied.

The proof of (\ref{fermi}) in \cite{phillips-sarnak2} uses the Lax-Phillips scattering theory as developed for automorphic functions, see \cite{lax-phillips}.
The crucial ingredient is provided by 
the cut-off wave operator $B$. Its spectrum (on appropriate spaces) coincides with the singular set (counting multiplicities). It includes  the embedded eigenvalues and the resonances. 
The motion of an embedded eigenvalue depending on  the perturbation
parameter $\e$ on the complex place $\C$ can be identified as the
motion of an eigenvalue of $B$. Given that Phillips and Sarnak proved
that regular perturbation theory applies to this setting, it follows that an embedded eigenvalue moves (with at most algebraic singularities) as function of $\e$, either remaining a cuspidal eigenvalue or becoming a resonance. Eq. (\ref{fermi}) follows using standard perturbation theory techniques. 
Balslev provided a different proof of Eq. (\ref{fermi})  in \cite{balslev} by introducing the technique of analytic dilations and imitating the setting of Fermi's Golden Rule for the helium atom, see \cite{simon}.
A slightly modified version of the application of perturbation theory is provided in \cite{petr}, using the formulas in \cite[p.~79]{kato}.

Once the dissolving condition had been identified, Phillips and Sarnak \cite{phillips-sarnak1} expressed it as a special value of a Rankin--Selberg convolution of $u_j$ with the holomorphic cusp form $f$ generating the deformation. These special values have been subsequently studied \cite{deshouillers-iwaniec, dips, luo} with the aim of showing that a generic surface with cusps has {\lq few\rq} embedded eigenvalues in the sense of Weyl's law.  

A different line of approach has been to develop alternate
perturbation settings, where the condition to check is easier to
understand. Wolpert, Phillips and Sarnak, and
Balslev and Venkov succeded in investigating Weyl's law this way. 
\cite{wolpert, phillips-sarnak3, balslev-venkov, balslev-venkov2}.

A more recent development came through the numerical investigation of
the poles of Eisenstein series by Avelin \cite{avelin}. Working with
the Teichm\"uller space of $\Gamma_0(5)$, she found a fourth order
contact of $s_j(\e)$ with the unitary axis $\Re (s)=1/2$.  It is easy
to explain why certain directions in the moduli space will not satisfy
the Phillips-Sarnak condition (\ref{phillipssarnakcondition}): If the
dimension of the moduli space is at least $2$, then the map $f\to
\langle \Delta^{(1)} u_j, E(z, 1/2+ir_j)\rangle $ is linear,
therefore, is has nontrivial kernel. Avelin also identified
numerically the most suitable curve that the singular point follows in
the left half-plane. This work (along with the work of Farmer and
Lemurell \cite{farmerlemurel}) motivated us to investigate  whether one can identify higher order Fermi-type conditions that will explain what happens in this case. We answer affirmatively:
we find conditions that guarantee that an embedded eigenvalue becomes
a resonance. 

For this purpose we introduce 
 the perturbation series of  the generalized eigenfunctions $D(z, s, \e)$, with $D(z, s, 0)=E(z, s)$:
 \begin{equation}
 D(z, s, \e)=D(z, s, 0)+D^{(1)}(z, s )\e+\frac{D^{(2)}(z, s)}{2!}\e^2+\cdots .\end{equation}
 \begin{theorem}\label{higher-fermi-intro}
 Assume that for $k=0, 1, \ldots , n-1$ the functions $D^{(k)}(z, s)$
 are regular at a simple cuspidal eigenvalue $s_j=1/2+ir_j$.
 Then $D^{(n)}(z, s)$ has at most a first order pole at $s_j$. 
{ \begin{enumerate}
\item If $D^{(n)}(z, s)$ has a pole at $s_j$, then the embedded eigenvalue becomes a resonance. \item Moreover, with $\norm{\cdot}$ the standard $L^2$-norm,
 \begin{equation*}
 \Re {s}_j^{(2n)}(0)=-\frac{1}{2}\binom{2n}{n}\norm{
   \res_{s=s_j}D^{(n)}(z, s)                }^2,\end{equation*}
and this is the leading term in the expansion of $\Re s_j(\e)$,
i.e. $\Re {s}_j^{(j)}(0)=0$ for $j<2n$.
\end{enumerate} }
\end{theorem}
 \begin{corollary}\label{cykel-intro}
 An embedded simple eigenvalue $s_j$ becomes a resonance if and only if for some $m\in \N$ the function $D^{(m)}(z, s)$ has a pole at $s_j$.
 \end{corollary}
\begin{remark}\label{phillips-sarnakcondition} For $n=1$ the condition in the theorem is the classical Fermi's Golden Rule, see (\ref{Dn-residue}) with $n=1$. Our method provides a new proof of this well-known result without using energy inner products, see \cite{phillips-sarnak2} but assuming Theorem \ref{singularsetproperties}.
\end{remark}
\begin{remark} The assumptions of the theorem may equivalently be
  stated as  $\Re(s^{(j)}(0))=0$ for $j=1,\ldots, 2n-1$. So in the
  theorem we  are really
  assuming that the embedded eigenvalue does not become a resonance to
  order \emph{less} than $2n$.
\end{remark}

\begin{remark} At first glance it may seem that the condition identifies one perturbation object with another, equally unknown. However, 
the condition can surprisingly also be expressed as the nonvanishing at a special point of a Dirichlet
series. The relevant series is more
complicated than the standard  Rankin--Selberg convolution. In the
case of character varieties and $n=2$ this Dirichlet series is 
\begin{equation}
\sum_{n=1}^{\infty} \left(  \sum_{k_1+k_2=n}\frac{a_{k_1}}{k_1}\frac{a_{k_2}}{k_2}             b_{-n}    \right)\frac{1 }{n^s},
\end{equation}
where $a_n$ are the Fourier coefficients of $f$, and $b_n$ are the coefficients of  $u_j$.

Even more  $D^{(n)}(z, s)$  has been the object of intense investigation by Goldfeld, O'Sullivan, Chinta, Diamantis,  the authors, Jorgenson et. al. \cite{goldfeld1, goldfeld2,  chinta, diamantis, osullivan, petridis, petridis-risager, jorgenson}.
It can be defined for $\Re (s)>1$ as
\begin{equation}\label{dseries}
D^{(n)}(z, s)=\sum_{\GinfmodG}\left(2\pi i \int_{i\infty}^{\gamma z} \Re f(w)\, dw\right)^n\Im(\gamma z)^s.
\end{equation}
 In fact, in \cite{osullivan2, petridis} it was proved that
$${  \rm Res}_{s=s_j}D^{(1)}(z, s)= \langle \Delta^{(1)} u_j, E(z, 1/2+ir_j)\rangle u_j(z),$$
 which gives the Phillips--Sarnak condition when one takes the
 $L^2$-norm.   This motivated us to investigate the residues of 
 $D^{(n)}(z, s)$ and derive  Theorem \ref{higher-fermi-intro}. The
 character perturbation setup  is analyzed in section \ref{character-varieties}.
\end{remark}

\begin{remark} The simplicity of $s_j$ is not important. We state  the
  theorem for any multiplicity of $s_j$ as Theorem \ref{maintheorem} in Section \ref{introductorymaterial}. 
\end{remark}
\begin{remark} This theorem gives an algorithmic method of checking
  whether in a particular direction of moduli space an embedded
  eigenvalue becomes a resonance. If $D^{(1)}(z, s)$ is regular at
  $s_j$, which is equivalent to the vanishing of the Phillips--Sarnak
  condition, then the embedded eigenvalue stays an eigenvalue to second order and we need to check the higher order condition $D^{(2)}(z, s)$. If this is regular one looks at the next term in the perturbation series of $D(z, s, \e)$ etc. 
\end{remark}

\begin{remark}
There is an easy argument that explains why a pole of $D^{(n)}(z, s)$ at $s_j$ forces the embedded eigenvalue to become a resonance.
The argument  is sketched in  section \ref{heustics-dissolving}.
\end{remark}

In the last section we investigate  whether cusp forms
could be stable under perturbations along certain paths in the
deformation space. Our main result is for groups with certain
symmetries e.g. (extensions of) $
\Gamma_0(p)$. Even-odd considerations guarantee the existence of Maa{\ss} cusp forms for the unperturbed problem.  In Theorem \ref{destruction-in-a-neighborhood} we
give sufficient conditions  to ensure that there is
\emph{no} path along which a cusp form remains. As a result
the cusp form is
isolated in the deformation space. If
the conditions are not satisfied, the cusp form  indeed remains in a
specific determined line in the deformation space.

In this article (Section \ref{character-varieties}) we investigate  the case of character varieties. The application of Theorem \ref{higher-fermi-intro} to the  analysis of Teichm\"uller deformations
will appear in \cite{petridis-risager2}.
 
\thanks We would like to thank D. Hejhal, A.  Str\"ombergsson,
  A. Venkov, E. Balslev, D. Mayer and P. Sarnak for helpful
  discussions and encouragements.

\section{Background and preliminaries}
\label{introductorymaterial}
An admissible surface, see  \cite{mueller2, mueller},  is a two dimensional non-compact Riemannian manifold $M$ of finite area with hyperbolic ends, i.e. there is a compact set $M_0$ such that $M$ has a decomposition  $$M=M_0\cup \bigcup_{{\mathfrak a}=1}^{\mathfrak k}Z_{\mathfrak a}$$
and
$$Z_{\mathfrak a}\cong S^{1}\times [c_{\mathfrak a}, \infty), \quad
c_{\mathfrak a}>0$$ carries coordinates $(x_{\mathfrak a}, y_{\mathfrak a})\in S^{1}\times
[c_{\mathfrak a}, \infty)$ and is 
equipped with the hyperbolic metric
$$\frac{dx_{\mathfrak a}^2+dy_{\mathfrak a}^2}{y_{\mathfrak a}^2}.$$
The end $Z_{\mathfrak a}$ is called a cusp. 
 M\"uller \cite{mueller2, mueller} has worked out the spectral theory of admissible surfaces. 
The Laplace operator $\Delta$, defined originally on compactly
supported smooth functions,  has a unique self-adjoint extension on $L^2(M)$, which we  denote by $L$.
The spectrum of $L$ consists of discrete spectrum (eigenvalues $\lambda_j=s_j(1-s_j)$) with 
$$0=\lambda_0<\lambda_1\le \lambda_2\le \ldots$$
(a finite or infinite set accumulating at $\infty$) and 
continuous spectrum $[1/4, \infty)$ of multiplicity $\mathfrak k$, provided by generalized eigenfunctions $E_{\mathfrak a}(z, s)$. These can to be constructed as in \cite{mueller, deverdiere} and, only in the special case of hyperbolic surfaces, are given by series of the type
$$E(z, s)=\sum_{\gamma\in \GinfmodG}\Im (\gamma z)^s.$$
We will call the generalized eigenfunctions Eisenstein series.

Each $E_{\mathfrak a}(z, s)$
 \begin{enumerate}\item admits  meromorphic continuation to $\C$ with poles in $\Re (s)<1/2$ or on the interval $(1/2, 1]$,
\item  satisfies the eigenvalue equation
$$\Delta E_{\mathfrak a}(z, s)+s(1-s)E_{\mathfrak a}(z, s)=0,\quad
\textrm{and }$$
\item satisfies the functional equation
\end{enumerate}
\begin{equation*}E_{\mathfrak a}(z, s)=\sum_{\mathfrak b =1}^{\mathfrak k}\phi_{\mathfrak a\mathfrak b}(s) E_{\mathfrak b}(z, 1-s)
\end{equation*}
for some functions $\phi_{\mathfrak a \mathfrak b}(s)$. The
determinant of the scattering matrix $\Phi (s)=(\phi_{\mathfrak a
  \mathfrak b}(s))_{\mathfrak a,\mathfrak b=1}^{\mathfrak k}$ is denoted $\phi (s)$.
The poles of $\phi (s)$ are called resonances.
The scattering matrix satisfies a functional equation  $\Phi(s)\Phi
(1-s)=I_{\mathfrak k\times \mathfrak k}$, and, moreover, 
\begin{equation}\label{phifunctional}\Phi (\bar s)=\overline{\Phi (s)}, \quad \Phi (s)^*=\Phi(\bar s).\end{equation}
The resolvent of the Laplace operator $R(s)=(\Delta+s(1-s))^{-1}$
defined on $L^2 (M)$ for $\Re (s)>1/2$, $s\not\in\hbox{spec}(L)$, admits a meromorphic continuation to $\C$, if we restrict the domain to a smaller space, e.g. $C_c^{\infty}(M)$, compactly supported functions on $M$.
The  limiting absorption principle holds:  i.e. the resolvent  kernel (Green's function) $r(z, z', s)$  satisfies
\begin{equation}\label{limit-absorption}
r(z, z', s)-r(z, z', 1-s)=\frac{1}{1-2s}\sum_{\mathfrak a=1}^{\mathfrak k}E_{\mathfrak a}(z, s)E_{\mathfrak a}(z', 1-s).\end{equation}
At a spectral point $s_j$ with eigenvalue $s_j(1-s_j)>1/4$ the resolvent kernel has a pole described by the Laurent expansion
\begin{equation}\label{resolventpolarstructure}r(z, z', s)=\frac{P}{s(1-s)-s_j(1-s_j)}+\cdots,\end{equation}
where $P$ is the spectral projection to the eigenspace with eigenvalue $s_j(1-s_j)$.

An admissible surface has generically finitely many discrete eigenvalues, a result due to Colin de Verdi\`ere
\cite{deverdiere}, and a consequence of the infinite dimensionality of the admissible metrics i.e. arbitrary metrics on $M_0$.
Determining the number of eigenvalues is much trickier if we demand
that $M$ is hyperbolic, as the Teichm\"uller space is finite dimensional. 
However, the perturbation setup works in the case of admissible surfaces and they appear as a technical devise in Teichm\"uller perturbations.

We are interested in perturbations of the Laplace operator $L$ on $M$.
The simplest kind of such arises from a perturbation of the Riemannian
metric inside $M_0$ (compact perturbations). Let $\e\in(-\e_0, \e_0)$. Let $g(\e)$ be a real
analytic family
of metrics on $M$, with $g(\e)=g(0)$ on $M\setminus
M_0$. The Laplacian then admits a real analytic expansion
$$L(\e)=L(0)+ \e L^{(1)}+\frac{\e^2}{2}L^{(2)}+\cdots .$$
As the family of metrics agree with $g(0)$ up in the cusps, the
Laplacian does not change up in the cusps and the operators $L^{(i)}$
are compactly supported operators, since $L^{(i)}f$ has support in
$M_0$ for every (smooth) function $f$.
We denote by $D(z,s,\e)$ any of the generalized eigenfunctions
$E_{\mathfrak{a}}(z,s,\e)$ of $L(\e)$.

\begin{theorem}\label{eisensteinregularity}
The family $D(z, s, \e)$ is real analytic in $\e$ for $\e\in(-\e_0, \e_0)$ and meromorphic in $s\in {\C}\setminus\{1/2\}$. 
The $n$-th derivative in $\e$ is given by
\begin{equation}\label{Dnresolvent}D^{(n)}(z, s)=-R(s)\sum_{i=1}^n\binom{ n}{ i} L^{(i)} D^{(n-i)}(z, s).\end{equation}
\end{theorem}
\begin{proof}[Sketch of proof]
The real analyticity follows from the construction of the generalized eigenfunctions $D(z, s, \e)$ using pseudo-Laplacians and the fact that the construction can be differentiated at every step in $\e$. This is explained (for the first derivative at least) in \cite{petridis}, see also \cite{bruggeman}. The formula for $D^{(n)}(z, s)$ can be proved by 
differentiating $(L(\e)+s(1-s))D(z, s, \e)=0$ to  get
$$\sum_{i=0}^n\binom{n}{i}\left.\frac{d^i}{d\e^i}(L(\e)+s(1-s))\right|_{\e=0}D^{(n-i)}(z, s)=0.$$
 Then we isolate  the term $D^{(n)}(z, s)$ using $R(s)$ for $\Re
 (s)>1/2$. Since the operators $L^{(i)}$ are compactly supported, the
 resolvent is applied to a compactly supported function and the
 right-hand side of (\ref{Dnresolvent}) can be meromorphically
 continued to $\C$.  The identity (\ref{Dnresolvent}) holds on $\C$ by
 the principle of analytic continuation. \end{proof}

If we expand $E_{\mathfrak a}(z, s)$ in the cusp $Z_{\mathfrak b}$,  the zero Fourier coefficient takes the form
$$\delta_{\mathfrak a\mathfrak b}y_{\mathfrak b}^s+\phi_{\mathfrak a \mathfrak b}(s)y_{\mathfrak b}^{1-s}.$$
It follows from Theorem \ref{eisensteinregularity}, that,  for $s\neq
1/2$,  $\phi_{\mathfrak a\mathfrak b}(s, \e)$ is also real analytic in $\e$, since
$$\phi_{\mathfrak a\mathfrak b}(s, \e)=\frac{1}{y^{1-s}_{\mathfrak b}}\left(
\int_0^1 E_{\mathfrak a}(z_{\mathfrak b}, s, \e)\,dx_{\mathfrak b}-\delta_{\mathfrak a\mathfrak b}y_{\mathfrak b}^s\right).$$

The singular set $\sigma$ includes the embedded eigenvalues and
resonances at the same time. The only points in the singular set $\sigma$ off the real axis
are \begin{enumerate} \item  $s_j$ with $s_j(1-s_j)$ an embedded
  eigenvalues counted with its multiplicity, 
and \item  resonances $s_j$ counted with multiplicity the order of the pole of the scattering determinant at $s_j$.\end{enumerate}
For the points in $[0,1]\cap\sigma$, see \cite{phillips-sarnak2}. A
point in the singular set is called singular.

The important theorem about the singular set needed is the following theorem:
\begin{theorem}{\cite[Corollary 5.2]{phillips-sarnak2}}\label{singularsetproperties}
If $s_j(0)$ is in the singular set $\sigma (0)$ for $\e=0$ and has multiplicity $1$, then it moves real analytically in $\e$ for $\abs{\e}$ sufficiently small. If the multiplicity is greater than one, then the singular points decompose into a finite system of real analytic functions having at most algebraic singularities.
\end{theorem} 
In the setting of compact perturbations of admissible surfaces, M\"uller \cite{mueller} proved the same statement.
These results use the family of cut-off wave operators $B(\e)$ and follow from standard perturbation theory, once it is proved that the resolvent $R_{B(\e)}(s)$ is real analytic for $\abs{\e}$ sufficiently small.
The technically difficult aspect of \cite{phillips-sarnak2} is the identification of the spectrum of $B(\e)$ with the singular set $\sigma (\e)$.

\section{Dissolving conditions of higher order} \label{proofs}

\subsection{Main statements}
 From this section onwards we restrict ourselves, for simplicity, to
the case that $M$ has one cusp, i.e. $\mathfrak{k}=1$.
 Let $s_j=s_j(0)$ be a singular point of multiplicity $m$.
Let $\hat{s}_j(\e) $ be the weighted mean of the branches of the singular points generated by splitting $s_j(0)$ under perturbation, i.e.
$$\hat{s}_j (\e)=\frac{1}{m}\sum_{l=1 }^m s_{j,l}(\e).$$
We are now ready to state and prove the more precise version of
Theorem \ref{higher-fermi-intro}: 

 \begin{theorem}\label{maintheorem}
 Assume that for $k=0, 1, \ldots , n-1$ the functions $D^{(k)}(z, s)$
 are regular at  a cuspidal eigenvalue $s_j=1/2+ir_j$.
 Then $D^{(n)}(z, s)$ has at most a first order pole at $s_j$. 
 \begin{enumerate}
\item  If $D^{(n)}(z, s)$ has a pole at $s_j$, then the embedded eigenvalue becomes a resonance.
\item  Moreover, with $\norm{\cdot}$ the standard $L^2$-norm,
 \begin{equation}
 \Re {\hat{s}}_j^{(2n)}(0)=-\frac{1}{2m}\binom{2n}{n}\norm{     \res_{s=s_j}D^{(n)}(z, s)                }^2.\end{equation}
\end{enumerate} 
\end{theorem}
 \begin{corollary}\label{cykel}
 At least one of the  cusp forms with given  $s_j$ becomes a resonance if and only if for some $m\in \N$ the function $D^{(m)}(z, s)$ has a pole at $s_j$.
 \end{corollary}

\subsection{Poles of $D^{(n)}(z,s)$, and dissolving cusp forms}\label{heustics-dissolving} Before we prove Theorem \ref{maintheorem}, we indicate an argument that explains why a singularity of $D^{(n)}(z, s)$ at $s_j$ is connected to 
dissolving cusp forms. For simplicity we consider hyperbolic surfaces
so that the generalized eigenfunction $D(z,s,\e)$ is an Eisenstein series. Let us assume that $D^{(k)}(z, s)$ is regular
at $s_j$ for $k=1, \ldots, n-1$. Assume $u_j$ is a simple cusp
form and that  $u_j(\e)$ remains a cusp form with $u_j(0)=u_j$. It is
known that cusp forms are perpendicular to the Eisenstein series $D(z,s,\e)$ for all $s$. This gives
\begin{equation}\label{easyeq}\inprod{u_j(\e)}{D(z, s, \e)}=0.\end{equation}
Phillips and Sarnak \cite{phillips-sarnak2} proved the real
analyticity of $u_j(\e)$. We differentiate (\ref{easyeq}) to get
$$\sum_{k=0}^{n}\binom{n}{k}\inprod{u_j^{(n-k)}}{D^{(k)}(z, s)}=0$$
for $s$ close to $s_j$. By the assumptions the term with $k=n$ should be a
regular function at $s_j$. Under the same assumptions, using
(\ref{Dnresolvent})  and (\ref{resolventpolarstructure}) we see that
$D^{(n)}(z, s)$ has at most a first order pole at $s_j$ with residue a multiple of $u_j(0)$. By regularity of $\inprod{u_j}{D^{(n)}(z, s)}$ this residue has to vanish.
  This approach does not prove Corollary \ref{cykel}
  but shows the sufficiency of the condition that some $D^{(n)}(z,s)$ has a pole to
  conclude that $s_j$ becomes a resonance. Corollary \ref{cykel} shows
  that this is also necessary.

\subsection{Polar structure of the Taylor coefficients of $\phi(s, \e)$}

Since the singular set is defined partly though the poles of
$\phi(s)$, we  now investigate the perturbation series of 
$\phi(s,\epsilon)$, in order to track the singular points as $\e$ varies.

The functional equation for $D(z, s, \e)$ is
\begin{equation}\label{eq:FE}D(z, s, \e)=\phi(s, \e)D(z, 1-s, \e).\end{equation} Since $D(z, s, \e)$ is real
analytic in $\e$ the same is true for $\phi(s,\epsilon)$ and we may introduce 
 the perturbation series of the scattering matrix $\phi(s, \e):$ 
 $$ \phi (s, \e)=\phi (s,
 0)+\phi^{(1)}(s)\e+\frac{\phi^{(2)}(s)}{2!}\e^2+\cdots . $$
We differentiate (\ref{eq:FE}) to identify the perturbation coefficients of $\phi (s, \e)$:
\begin{equation}\label{dnfunctional}D^{(n)}(z, s)=\sum_{i=0}^n\binom{n}{i} \phi^{(i)}(s)D^{(n-i)}(z, 1-s).\end{equation}

\begin{proposition}\label{scatteringphin}
The perturbation coefficients of the scattering matrix are given by
$$\phi^{(n)}(s)=\frac{1}{2s-1}\int_ME(z, s)\sum_{i=1}^{n}\binom{n}{i}L^{(i)}D^{(n-i)}(z, s)\, d\mu (z), \quad n\ge 1.$$
\end{proposition}
\begin{proof}
The proof is already in \cite{risagerthesis}. We include the argument here.
We proceed by induction.   By using first (\ref{Dnresolvent})  and
then (\ref{limit-absorption})  we find that  
\begin{align*}D^{(1)}&(z, s)=-R(s)L^{(1)}E(z, s)\\
&=-R(1-s)L^{(1)}E(z, s)+\frac{1}{2s-1}\int_M E(z', s)L^{(1)}E(z', s)\,
d\mu (z') E(z, 1-s)\\
&=\phi(s)(-R(1-s)L^{(1)}E(z, 1-s))+\frac{1}{2s-1}\int_M E(z', s)L^{(1)}E(z', s)\,
d\mu (z') E(z, 1-s)\\
&=\phi(s)D^{(1)}(z,1-s)+\frac{1}{2s-1}\int_M E(z', s)L^{(1)}E(z', s)\,
d\mu (z')E(z, 1-s).
\end{align*}
From  (\ref{dnfunctional}) we know that 
$$D^{(1)}(z, s)=\phi^{(1)}(s)E(z, 1-s)+\phi(s) D^{(1)}(z, 1-s),$$
and since $E(z,1-s)$ does not vanish identically, we get the result for $n=1$.

Assume the formula has been proved for $m<n$. Using (\ref{Dnresolvent}) and (\ref{limit-absorption}) we get
\begin{eqnarray*}
{D^{(n)}(z, s)}&=&-R(1-s)\sum_{i=1}^n\binom{ n}{ i} L^{(i)} D^{(n-i)}(z, s)\\&&+\frac{1}{2s-1}\left(\int_ME(z, s)\sum_{i=1}^n\binom{ n}{ i} L^{(i)} D^{(n-i)}(z, s)\, d\mu (z) \right)E(z, 1-s)\\
&=&-R(1-s)\sum_{i=1}^n\binom{n}{ i} L^{(i)}\sum_{k=0}^{n-i}\binom{n-i}{ k}\phi^{(k)}(s)D^{(n-i-k)}(z, 1-s)+Q(z, s)\\
&=&-R(1-s)\sum_{k=0}^{n-1}\phi^{(k)}(s)\sum_{i=1}^{n-k}\binom{n}{ k}\binom{ n-k}{ i}L^{(i)}D^{(n-i-k)}(z, 1-s)+Q(z, s)\\
&=&\sum_{k=0}^{n-1}\binom{n}{ k}\phi^{(k)}(s)\left(-R(1-s)\sum_{i=1}^{n-k}\binom{ n-k }{ i}L^{(i)}D^{(n-k-i)}(z, 1-s)\right)+Q(z, s)\\
&=&\sum_{k=0}^{n-1}\binom{ n }{ k}\phi^{(k)}(s)D^{(n-k)}(z, 1-s)+Q(z, s),
\end{eqnarray*}
where 
$$Q(z, s)= \frac{1}{2s-1}\int_ME(z', s)\sum_{i=1}^{n}\binom{n}{i}L^{(i)}D^{(n-i)}(z', s)\, d\mu (z') E(z, 1-s).$$
Comparing with (\ref{dnfunctional}), we get that
$$\phi^{(n)}(s)=\frac{1}{2s-1}\int_ME(z, s)\sum_{i=1}^{n}\binom{n}{i}L^{(i)}D^{(n-i)}(z, s)\, d\mu (z),$$
which finishes the proof.
 \end{proof}
Proposition \ref{scatteringphin} allows to recover all the scattering
terms in terms of the perturbed Eisenstein series. However, for
$\phi^{(n)}(s)$ one uses information for $D^{(j)}(z, s)$ with $j$ up
to $n$. For our purposes this is \emph{not} good enough. The following
technical yet important proposition allows to use fewer $D^{(j)}(z, s)$.
\begin{proposition}\label{phireduction-result} 
The perturbed terms of the scattering function $\phi^{(n)}(s)$ are given for $i=1, 2, \ldots , n-1$ by
\begin{align}\label{phireduction}
\nonumber (2s-1)\phi^{(n)}(s)&=\binom{n}{i}\inprod{
  \sum_{k=1}^{n-i}\binom{n-i}{k}L^{(k)}D^{(n-i-k)}(z, s)}{D^{(i)}(z,
  \bar s)}\\ & \qquad +\sum_{k=i+1}^n\binom{n}{k}\inprod{D^{(n-k)}(z, s)}{ \sum_{m=0}^{i-1}\binom{k}{m}L^{(k-m)}D^{(m)}(z, \bar s)}.\end{align}
\end{proposition}
\begin{proof}
To simplify the notation we suppress $z$ and $s$ and $\bar s$ in the inner products. It will be understood that the terms on the left of $\langle \cdot, \cdot \rangle $ should carry $s$ and the one on the right should have $\bar s$. Note that for any functions $f, g$ we have $\inprod{R(s) f}{g}=\inprod{f}{ R(\bar s) g}$, since $R(s)^*=R(\bar s)$.
Moreover, since $L(\e)$ are self-adjoint, the same applies to $L^{(j)}$.
Even if the Eisenstein series are not in $L^2$, since $L^{(j)}$ are compactly supported, we can easily justify the integration by parts in the following calculation. 
We have from Proposition \ref{scatteringphin} and (\ref{Dnresolvent})
\begin{align*}(2s-1)&\phi^{(n)}(s)=\inprod{\sum_{i=1}^{n}\binom{n}{i}L^{(i)}D^{(n-i)}}{E}\\&=\binom{n}{1}\inprod{ D^{(n-1)}}{L^{(1)}E}+\sum_{i=2}^n\binom{n}{i}\inprod{D^{(n-i)}}{ L^{(i)}E}\\
\allowdisplaybreaks &=\binom{n}{1}\inprod{\sum_{k=1}^{n-1}\binom{n-1}{k}(-R)L^{(k)} D^{(n-1-k)}}{L^{(1)}E}+\sum_{i=2}^n\binom{n}{i}\inprod{D^{(n-i)}}{ L^{(i)}E}
\\
\allowdisplaybreaks &=\binom{n}{1}\inprod{\sum_{k=1}^{n-1}\binom{n-1}{k}L^{(k)} D^{(n-1-k)}}{-RL^{(1)}E}+\sum_{i=2}^n\binom{n}{i}\inprod{D^{(n-i)}}{ L^{(i)}E}
\\
\allowdisplaybreaks &=\binom{n}{1}\inprod{\sum_{k=1}^{n-1}\binom{n-1}{k}L^{(k)} D^{(n-1-k)}}{D^{(1)}}+\sum_{k=2}^n\binom{n}{k}\inprod{D^{(n-k)}}{ L^{(k)}E}.
\end{align*}
This shows  (\ref{phireduction}) for $i=1$. Assume now that we proved it for a given $i$. We separate the terms with $k=1$ and $k=i+1$ in (\ref{phireduction}) and group them together to get
\begin{align*}
 (2s-1)\phi^{(n)}(s)&=\binom{n}{i}\inprod{ \binom{n-i}{1}L^{(1)}D^{(n-i-1)}}{D^{(i)}}\\ 
 \allowdisplaybreaks &+\binom{n}{i+1}\inprod{D^{(n-(i+1))}}{\sum_{m=0}^{i-1}\binom{i+1}{m}L^{(i+1-m)}D^{(m)}}\\
\allowdisplaybreaks  &+\binom{n}{i}\inprod{\sum_{k=2}^{n-i}\binom{n-i}{k}L^{(k)}D^{(n-i-k)}}{D^{(i)}}\\&+
 \sum_{k=i+2}^n\binom{n}{k}\inprod{D^{(n-k)}}{ \sum_{m=0}^{i-1}\binom{k}{m}L^{(k-m)}D^{(m)}}.
 \end{align*}
 We use the obvious identity for binomial coefficients
 $$\binom{n}{i}\binom{n-i}{1}=\binom{n}{i+1}\cdot (i+1)$$
 and bump the summation variable by $i$ in the third sum to get
 \begin{align*}
& \binom{n}{i+1}\inprod{D^{(n-(i+1))}}{(i+1)L^{(1)}D^{(i)}+\sum_{m=0}^{i-1}\binom{i+1}{m}L^{(i+1-m)}D^{(m)}}\\
 &+\sum_{k=i+2}^n\binom{n}{i}\binom{n-i}{k-i}\inprod{L^{(k-i)}D^{(n-k)}}{D^{(i)}}+
 \sum_{k=i+2}^n\binom{n}{k}\inprod{D^{(n-k)}}{ \sum_{m=0}^{i-1}\binom{k}{m}L^{(k-m)}D^{(m)}}.
  \end{align*}
  We use
  $$ \binom{n}{i}\binom{n-i}{k-i}=\binom{n}{k}\binom{k}{i}$$ and (\ref{Dnresolvent}) to see that the expression is now
 \begin{align*}
 &\binom{n}{i+1}\inprod{\sum_{k=1}^{n-i-1}\binom{n-i-1}{k}L^{(k)}D^{(n-i-1-k)}}{-R\left(\sum_{m=0}^{i}\binom{i+1}{m}L^{(i+1-m)}D^{(m)}\right)}
 \\
& +\sum_{k=i+2}^n\binom{n}{k}\binom{k}{i}\inprod{L^{(k-i)}D^{(n-k)}}{D^{(i)}}+
 \sum_{k=i+2}^n\binom{n}{k}\inprod{D^{(n-k)}}{ \sum_{m=0}^{i-1}\binom{k}{m}L^{(k-m)}D^{(m)}}.\end{align*}
 We use (\ref{Dnresolvent}) again to get \begin{align*}
 &\binom{n}{i+1}\inprod{\sum_{k=1}^{n-i-1}\binom{n-i-1}{k}L^{(k)}D^{(n-i-1-k)}}{D^{(i+1)}}\\
 &+\sum_{k=i+2}^n\binom{n}{k}\left(\binom{k}{i}\inprod{L^{(k-i)}D^{(n-k)}}{D^{(i)}}+
 \inprod{D^{(n-k)}}{ \sum_{m=0}^{i-1}\binom{k}{m}L^{(k-m)}D^{(m)}}\right).\end{align*}
Finally we get \begin{align*}
(2s-1)\phi^{(n)}(s)=& \binom{n}{i+1}\inprod{\sum_{k=1}^{n-i-1}\binom{n-i-1}{k}L^{(k)}D^{(n-i-1-k)}}{D^{(i+1)}}\\&+
 \sum_{k=i+2}^n\binom{n}{k}\inprod{D^{(n-k)}}{ \sum_{m=0}^{i}\binom{k}{m}L^{(k-m)}D^{(m)}}.\\
 \end{align*}
\end{proof}

We can now use Proposition \ref{phireduction} to translate information
about $D^{(i)}(z,s)$ at $s_j$ into information about $\phi^{(k)}(s)$ at $s_j$:

\begin{theorem}\label{orderpolephi(s)} Assume that $D^{(q)}(z,s)$ is regular at $s_j=1/2+ir_j$ for
  $q=0,1, \ldots, n-1$. Then 
\begin{enumerate}
\item \label{one} the function $\phi^{(l)}(s)$ is regular at $s_j$ for $l=0, 1,\ldots 2n-1$.
\item  \label{two} the function $\phi^{(2n)}(s)$ has at most a simple pole at $s_j$. Furthermore
  the residue at $s_j$ is given by
  \begin{equation*}
    \res_{s=s_j}\phi^{(2n)}(s)= -\phi(s_j)\binom{2n}{n}\norm{\res_{s=s_j}D^{(n)}(z,s)}^2.
  \end{equation*}

\end{enumerate}

\end{theorem}
\begin{proof} We  take
  $n=l$ for the various values of $l\leq 2n$ in Proposition \ref{phireduction-result}.
Assume first that $l<2n$ and let in Proposition
\ref{phireduction-result}   the integer $i$ be the integral part of $l/2$. Then $i<n$
and $l-i-1<n$ and therefore, by the assumption on $D^{(q)}(z,s)$ for
$q=0,1, \ldots, n-1$, we see immediately - using the expression on the right of
(\ref{phireduction}) - that $\phi^{(l)}(s)$ is regular
at $s_j$.

To prove the claim about $\phi^{(2n)}(s)$ we choose the integer $i$  in Proposition
\ref{phireduction-result} to equal $n$.  By (\ref{phireduction}) and the
assumptions on  $D^{(q)}(z,s)$ we see that $\phi^{(2n)}(s)$ has at most a simple pole at $s_j$ and that the residue is given by  
\begin{equation}
  \label{good-expression}
  \frac{1}{2s_j-1}\binom{2n}{n}\int_M{
    \sum_{k=1}^{n}\binom{n}{k}L^{(k)}D^{(n-k)}(z,
    s_j)}{\res_{s=s_j}D^{(n)}(z, s)}d\mu(z).
\end{equation}
From (\ref{dnfunctional}) and the assumptions on $D^{(q)}(z,s)$ we get
that
\begin{equation*}
  \res_{s=s_j}D^{(n)}(z,s_j)=\phi(s_j)\res_{s=s_j}D^{(n)}(z,1-s).
\end{equation*}
Since $D^{(n)}(z,s)=\overline{D^{(n)}(z,\overline{s})}$ we have also,
since $1-\overline{s_j}=s_j$, that 
\begin{equation*}
\res_{s=s_j}D^{(n)}(z,1-s)=-\overline{\res_{s=s_j}D^{(n)}(z,s)},  
\end{equation*} and therefore
\begin{equation}\label{get-to-conjugate}\res_{s=s_j}D^{(n)}(z,s)=-\phi(s_j)\overline{\res_{s=s_j}D^{(n)}(z,s)}.  \end{equation}
From (\ref{Dnresolvent}) and (\ref{resolventpolarstructure}) we see that
\begin{equation}\label{Dn-residue}
  \res_{s=s_j}D^{(n)}(z,s)=\frac{1}{2s_j-1}\sum_{i=1}^m
  \inprod{\sum_{k=1}^n\binom{n}{k}L^{(k)}D^{(n-k)}(z,s_j)}{u_{j,i}}u_{j,i},
\end{equation} where $\{u_{j,i}\}$ is an orthonormal basis for  the eigenspace
for the eigenvalue $s_j(1-s_j)$. Inserting this in
(\ref{good-expression}) (after first using (\ref{get-to-conjugate})) we find that
$\res_{s=s_j}\phi^{(2n)}(s)$ is given by 
\begin{equation}\
-\frac{\phi(s_j)}{\abs{2s_j-1}^2}\binom{2n}{n}\sum_{i=1}^m\abs{\inprod{\sum_{k=1}^n\binom{n}{k}L^{(k)}D^{(n-k)}(z,s_j)}{u_{j,i}}}^2
\end{equation}
 which is easily seen to be the claimed result comparing
(\ref{Dn-residue}). 
\end{proof}
\begin{remark}
We need one more ingredient about $\phi(s,\e)$ before proving Theorem~\ref{maintheorem}. Since $\phi (s, \e)=\overline{\phi (\bar s, \e)}$
we deduce that
\begin{equation}\label{phiprimeoverphi}\frac{\phi'(s, \e)}{\phi(s, \e)}=\overline{\left(\frac{\phi' (\bar s, \e)}{\phi (\bar s, \e)}\right)},\end{equation} where $'$ denotes derivative in the $s$ variable, as is standard in the Selberg theory of the trace formula.
This follows from the fact that for an analytic function $f$ we have $$\frac{d}{ds}\overline {f(\bar
  s)}=\overline{f'(\bar s)}.$$  
\end{remark}
\begin{proof}[Proof of Theorem~\ref{maintheorem}]
We want to track the movement of the embedded eigenvalue/resonance in the
left half-plane. We define $\Gamma$ to be the semicircular contour $\gamma_1 (t)=ue^{it}+s_j$, $\pi/2\le t\le 3\pi/2$ followed by the vertical segment $\gamma_2 (t)=s_j+it$, $ -u\le t\le u$. Here   $u$ is chosen small enough, so that the only singular point for $\e=0$ inside  the ball $B(s_j, u)$ is $s_j$ with multiplicity $m=m(s_j)$. This contour is traversed counterclockwise.
For small enough $\e$ the total multiplicities of  the singular points $s_j(\e)$ inside $B(s_j, u)$ is $m(s_j)$. Perturbation theory allows to study the weighted mean $\hat{s}(\e)$ of the branches of eigenvalues of $B(\e)$. We have
\begin{equation}\label{blah}m(\hat{s} (\e)-s_j)=-\frac{1}{2\pi i}\int_{\Gamma} (s-s_j)\frac{\phi' (s, \e)}{\phi (s, \e)}\, ds+\sum_{j\in C}(s_j(\e)-s_j),\end{equation}
where $C$ is indexing the cusp forms eigenbranches inside $B(s_j, u)$, i.e. the cusp forms that remain cusp forms. Let the last sum be denoted by $p(\e)$. The reason for using $\Gamma$ and not the whole $\partial B(s_j, u)$ is that on the right half-disc $\phi (\e)$ has zeros, which we do not want to count.
Notice that $\overline{\int_{\gamma}f(s)\, ds}=\int_{\bar \gamma} \bar
f(\bar s)\, ds$ and, therefore, by (\ref{phiprimeoverphi})
$$m(\overline{\hat{s} (\e)-s_j})=\frac{1}{2\pi i}\int_{\bar \Gamma} (s-\bar s_j)\overline{\left(\frac{\phi' (\bar s, \e)}{\phi (\bar s, \e)}\right)}\, ds+\overline{p(\e)}=\frac{1}{2\pi i}\int_{\bar \Gamma}(s-\bar s_j)\frac{\phi' (s, \e)}{\phi (s, \e)}\, ds+\overline{p(\e)}.$$
Denoting by $-\gamma$ the contour $\gamma$ traversed in the opposite direction, we get
\begin{align*}m(\overline{\hat{s} (\e) -s_j})&=-\frac{1}{2\pi i }\int_{-\bar \Gamma}(s-\bar s_j)\frac{\phi'(s, \e)}{\phi (s, \e)}\, ds+\overline{p(\e)}\\&=-\frac{1}{2\pi i}\int_{T^{-1}(-\bar\Gamma)}(1-w-\bar s_j)\frac{\phi'(1-w, \e)}{\phi (1-w, \e)}\, (-dw)+\overline{p(\e)},\end{align*}
where $s=T(w)=1-w$ is a conformal map. By the functional equation
$\phi (s, \e)\phi (1-s, \e)=1$, see (\ref{phifunctional}), we get
$$\phi'(s, \e)\phi (s, \e)-\phi (s, \e)\phi'(1-s, \e)=0,$$ which implies $$\frac{\phi'(s, \e)}{\phi (s, \e)}=\frac{\phi'(1-s, \e)}{\phi(1-s, \e)}.$$
We plug this into the expression for $m(\overline{\hat{s} (\e) -s_j})$ to get
\begin{equation}\label{blahblahblah}m(\overline{\hat{s} (\e) -s_j})=-\frac{1}{2\pi i}\int_{T^{-1}(-\bar\Gamma)}(w- s_j)\frac{\phi'(w, \e)}{\phi (w, \e)}\, dw+\overline{p(\e)}. \end{equation}
We sum (\ref{blah}) and (\ref{blahblahblah}) and notice that the
cuspidal branch contributions cancel, because for a cuspidal branch
$s_{j,l}(\e)$ the function  $s_{j,l}(\e)-s_j$ is purely imaginary.
 We deduce that
 \begin{eqnarray}\label{flocke} \nonumber 2m\Re (\hat{s}(\e)-s_j)&=&-\frac{1}{2\pi i}\int_{\Gamma+T^{-1}(-\bar \Gamma)}(s-s_j)\frac{\phi'(s, \e)}{\phi(s, \e)}\, ds\\&=&  -\frac{1}{2\pi i}\int_{\partial B(s_j,
 u)} (s-s_j)\frac{\phi'(s, \e)}{\phi(s, \e)}\, ds,\end{eqnarray}
 since the contribution from the line
 segment on $\Re(s)=1/2$ from $\Gamma$ and $T^{-1}(-\bar \Gamma)$ cancel.
  By uniform convergence 
 we can differentiate the last formula in $\e$. We get
 \begin{align}\label{mosedebatteri}&\left.2m\frac{d^{2n}}{d\e^{2n}}\Re (\hat{s} (\e))\right|_{\e=0}=-\frac{1}{2\pi i}\int_{\partial B(s_j, u)}(s-s_j)\frac{d^{2n}}{d\e^{2n}}\left.\left(\frac{\phi'(s, \e)}{\phi (s, \e)}\right)\right|_{\e=0}\, ds\\
\nonumber &=-\frac{1}{2\pi i}\int_{\partial B(s_j, u)}(s-s_j)\sum_{k=0}^{2n}\binom{2n}{k} \left.\frac{d^{k}\phi'(s, \e)}{d\e^{k}}\right|_{\e=0}\left.\frac{d^{2n-k}(\phi (s, \e)^{-1})}{d\e^{2n-k}}\right|_{\e=0}\, ds.
\end{align}
 We can interchange the order of differentiation 
 $$\frac{d^k}{d\e^k}\phi'(s, \e)=\frac{d}{ds}\phi^{(k)}(s)$$ and see that
 this is regular at $s_j$ for $k<2n$ by Theorem \ref{orderpolephi(s)}.
 On the other hand for $k=2n$ it has a double pole at $s_j$ by the same theorem. 
 Concerning
 \begin{equation*}
   \frac{d^{2n-k}}{d\e^{2n-k}}\phi(s,\e)^{-1}
 \end{equation*}
we argue as follows: We  differentiate $m$ times $\phi(s,\e)^{-1}\phi(s,\e)=1$ to get
\begin{equation*}
  \sum_{k=0}^m\binom{m}{k}\left.\frac{d^{k}}{d\e^{k}}\phi(s,\e)^{-1}\right|_{\e=0}\phi^{(m-k)}(s,0)=0. 
\end{equation*}
Let $m$ be less than $2n$. By Theorem \ref{orderpolephi(s)}, the fact
that $\phi(s)$ is unitary on $\Re(s)=1/2$, and by
solving for $\left.\frac{d^{m}}{d\e^{m}}\phi(s,\e)^{-1}\right|_{\e=0}$,
 we see that 
$\left.\frac{d^{m}}{d\e^{m}}\phi(s,\e)^{-1}\right|_{\e=0}$ is regular
at $s_j$. For $m=2n$ we see, by the same
argument, that
$\left.\frac{d^{2n}}{d\e^{2n}}\phi(s,\e)^{-1}\right|_{\e=0}$ has at most
a simple pole at $s_j$.

We can now determine the order of the pole of the integrand of the
right-hand side in (\ref{mosedebatteri}): By the above considerations
we see that the only non-regular term occurs for $k=2n$. This  is the term
 \begin{equation*}
  (s-s_j)\frac{d \phi^{(2n)}(s,0)}{ds}\phi^{-1}(s),
\end{equation*}
which has at most a simple pole. By the residue theorem the expression in (\ref{mosedebatteri})  equals minus the residue of
\begin{equation*}(s-s_j)\frac{d \phi^{(2n)}(s,0)}{ds}\phi^{-1}(s).\end{equation*} Since the leading
term in the Laurent expansion of the derivative in $s$ of  $\phi^{(2n)}(s,0)$
equals $-(\res_{s=s_j}\phi^{(2n)}(s,0))/(s-s_j)^2$ we conclude
that 
\begin{align*}
  \left.2m\frac{d^{2n}}{d\e^{2n}}\Re (\hat{s}
    (\e))\right|_{\e=0}&=\frac{\res_{s=s_j}\phi^{(2n)}(s,0)}{\phi(s_j,0)}\\
&=-\binom{2n}{n}\norm{\res_{s=s_j}D^{(n)}(z,s)}^2,
\end{align*}
where, in the last equality, we used Theorem \ref{orderpolephi(s)} again.
This completes the proof of the theorem.\end{proof}

\begin{proof}[Proof of Corollary \ref{cykel}]
The direction that a pole of some $D^{(m)}(z, s)$ at $s_j$ implies
that at least one embedded eigenvalue becomes a resonance is proved as
follows: If $n$ is the smallest number such that $D^{(n)}(z, s)$ has a
pole at $s_j$, then from Theorem \ref{maintheorem}, we have that $\Re
\hat{s}_j^{(2n)}\ne 0$, while $\Re \hat{s}_j^{(2m)}=0$ for $m<n$. If
$k$ is the smallest integer with $\Re\hat{s}_j^{(k)}\ne 0$, then
$k=2n$, since an odd leading term in the Taylor series of $\Re
\hat{s}_j(\e)$ will force $\Re\hat{s}_j(\e)$ to take values larger and
smaller than $1/2$. This is impossible, since a singular point cannot
move to the right half-plane. Therefore $\hat {s}_j(\e)$ does not have
real part equal to $1/2$ for all small $\e$ and one of the cuspidal eigenvalues has
to dissolve.
The opposite direction is obvious: If all embedded eigenvalues remain embedded eigenvalues, then $\Re \hat{s}_j(\e)=1/2.$ This implies that $\Re \hat{s}_j^{(2n)}=0$ for all $n\in \N$.

\end{proof}


\section{Character varieties}\label{character-varieties}

\subsection{Higher order
  dissolving for character
  varieties}\label{sec:dissolving-character-case} We now describe how
the above theory can be modified for the twisted spectral problem
related to character varieties. Let $\G$ be a
discrete cofinite subgroup of $\psl$ with quotient $M=\GmodH$, where
$\H$ is the upper half-plane. For simplicity, we still  assume that
$\Gamma$ has precisely one cusp, which we
assume is at infinity. Let $f(z)\in S_2(\Gamma)$ be a holomorphic cusp form of weight $2$. Then $\omega=\Re (f(z)\, dz) $ and $\omega=\Im (f(z)\, dz)$ are harmonic cuspidal 1-forms.  Let $\alpha $ be a compactly supported $1$-form in the same cohomology class as one of them. For the exact construction see e.g. \cite[Prop.~2.1]{petridis-risager}.  We fix $z_0\in\H$.
Define a family of characters 
$$ \begin{array}{lccc}\chi(\cdot,\e):&\Gamma&\rightarrow& S^1\\
                   &\g&\mapsto &\exp(-2\pi i\e\int_{z_0}^{\gamma z_0}\alpha).
\end{array}
$$
We consider the space
$$\L(\GmodH,\overline \chi(\cdot,\epsilon))$$ of
$(\G,\overline \chi(\cdot,\epsilon))$-automorphic functions, i.e. functions
$f:\H\to\C$ where $$f(\g z)=\overline{\chi}(\g,\e)f(z),$$ and
$$\int_{\GmodH}\abs{f(z)}^2d\mu(z)<\infty.$$  The automorphic Laplacian $\tilde L(\e)$ is
the closure of the operator acting on smooth functions in
$\L(\GmodH,\overline \chi(\cdot,\epsilon))$ by $\Delta f$. We denote
its resolvent by $\tilde R(s,\e)=(\tilde L(\e)+s(1-s))^{-1}$.
We introduce unitary operators
\begin{equation}\begin{array}{rccc}
U(\e):&\L(\GmodH)&\to& \L(\GmodH,\overline \chi(\cdot,\epsilon))\\
&f&\mapsto&\exp\left(2\pi i\e\int_{z_0}^z\alpha\right)f(z). 
\end{array}
\end{equation}
We then define
\begin{eqnarray}
\label{twisted-laplacian-1v}L(\e)&=&U^{-1}(\e)\tilde L(\e)U(\e)\\
R(s,\e)&=&U^{-1}(\e)\tilde R(s,\e)U(\e). 
\end{eqnarray}
The operators $L(\e)$ on $\L(\GmodH)$ and $\tilde L(\e)$ on $
\L(\GmodH,\overline \chi(\cdot,\epsilon))$ are unitarily equivalent. 
Notice  that $L(\e)$ and $R(s,\e)$ act on the fixed space
$\L(\GmodH)$, which allows to apply perturbation theory. It is easy to verify that 
\begin{align}\label{muffinstogo}L(\e)h=\Delta h +4\pi i\e\modsym{dh}{\alpha}&-2\pi i\e\delta(\alpha)h-4\pi^2\e^ 2\modsym{\a}{\a}\\
\label{coffeetogo}(L(\e)+s(1-s))R(s,\e)=&R(s,\e)(L(\e)+s(1-s))=I.
\end{align}
Here \begin{eqnarray*}\modsym{f_1dz+f_2d\overline z}{g_1dz+g_2d\overline z}&=&2y^ 2(f_1\overline{g_1}+f_2\overline{g_2})\\
\delta(pdx+qdy)&=&-y^2(p_x+q_y).
\end{eqnarray*}
 We notice also that 
\begin{align}
\label{firstLvariation} L^{(1)}(\e)h&=4\pi i\modsym{dh}{\a}-2\pi i\delta (\alpha)h-8\pi^2 \e\modsym{\a}{\a},\\ \allowdisplaybreaks
L^{(2)}(\e)h&=-8\pi^2 \modsym{\a}{\a},\\\allowdisplaybreaks
L^{(i)}(\e)h&=0,\quad\textrm{ when }i\geq 3.
\end{align}

We notice that $L^{(i)}$ are compactly supported operators and that $\delta(\omega)=0$ for a harmonic form $\omega$. 

We let $E(z,s,\epsilon)$ be the usual Eisenstein series for the system
$(\Gamma,\overline{\chi}(\cdot,\e))$ and define the $\Gamma$-invariant function 
$$D(z,s,\e)=U^{-1}(\e)E(z,s,\e).$$
The Phillips--Sarnak condition for dissolving cusp forms in this setting is:
\begin{equation}\inprod{L^{(1)} u_j}{E(z, s_j)}\ne 0.\end{equation}
The family of operators $L(\e)$ do not arise from an admissible
metric, but all the properties described in Section
\ref{introductorymaterial} are well-known, and the proof of the
dissolving theorem carries over almost verbatim, so in
this case the higher order analogue of Fermi's golden rule also holds:

\begin{theorem}\label{maintheorem-for-characters}
 Assume that for $k=0, 1, \ldots , n-1$ the functions $D^{(k)}(z, s)$
 are regular  close to a cuspidal eigenvalue $s_j=1/2+ir_j$.
 Then $D^{(n)}(z, s)$ has at most a first order pole at $s_j$. 
 \begin{enumerate}
\item  If $D^{(n)}(z, s)$ has a pole at $s_j$, then the embedded eigenvalue becomes a resonance.\item  Moreover, with $\norm{\cdot}$ the standard $L^2$-norm,
 \begin{equation}
 \Re {\hat{s}}_j^{(2n)}(0)=-\frac{1}{2m}\binom{2n}{n}\norm{     \res_{s=s_j}D^{(n)}(z, s)                }^2.\end{equation}
\end{enumerate} 
\end{theorem}


\subsection{Multiparameter perturbations}\label{multivariate}
To analyze the higher order dissolving conditions and relate it to
 a Dirichlet series it is useful  to 
work with families of characters depending on several parameters.
To do this we introduce the following notation. Given $\alpha_l$,
$l=1, \ldots, k$ harmonic, 
compactly supported 1-forms on $M$, we let ${\underline{\alpha}}=(\alpha_1,\ldots,\alpha_k)$, ${\underline{\e}}=(\e_1, \ldots, \e_k)$ and define
\begin{equation}\label{d-alphaseries}
D(z, s, \underline \alpha)=\sum_{\gamma\in \GinfmodG}\prod_{l=1}^k \left(\int_{i\infty}^{\gamma z}\alpha_l\right)
 \Im (\gamma z)^s.
\end{equation}
Such series have been studied in \cite[Lemma~2.4]{petridis-risager},
and we give a quick review of some of their properties: 

Let \begin{equation}\label{karakter}\chi(\gamma,\underline \e)=\prod_{l=1}^k{\exp\left(-2\pi i\e_l\int_{z_0}^{\gamma z_0}\alpha_l\right)}.
\end{equation} be the multiparameter character induced from $\underline\alpha$. We
know from the theory of Eisenstein series (See e.g. \cite{selberg,
  hejhal, iwaniec}) that $$E(z,s,\underline \e)=\sum_{\gamma\in
  \GinfmodG}\chi(\gamma,\underline\e)\Im (\gamma z)^s,\quad \Re(s)>1.$$
admits meromorphic continuation to $\C$ and that it satisfies a functional equation 
\begin{equation}\label{functional-equation-E}E(z,s,\underline \e)=\phi(s,\underline
\e)E(z,1-s,\underline\e).\end{equation} If we let 
\begin{equation}
U(\underline \e)f=\prod_{l=1}^k\exp\left(2\pi i\e_l\int_{z_0}^z\alpha_l\right)f(z)
\end{equation}
we see that when $\Re(s)>1$
\begin{equation}\label{original-petridis-obs}D(z, s, \underline{\alpha})=\left.\frac{\partial^k}{\partial
  \e_1\cdots\partial\e_k}U(-\underline \e)E(z,s,\underline \e)\right\vert
_{\underline \e=\underline 0}.\end{equation}
We have -- analogous to the 1-parameter situation described in the beginning of
this section --  that if $L(\underline\e)=U^{-1}(\underline \e)\tilde
L(\underline \e)U(\underline \e)$  then  
\begin{equation}\begin{array}{rcl}
L(\underline \e)h&=&\Delta h+4\pi i \sum_{l=1}^{k}\epsilon_l\langle dh, \alpha_k\rangle-2\pi i \left(\sum_{l=1}^{k}\e_l \delta(\alpha_k)\right) h
 \\ &&-4\pi ^2 \left(\sum_{l,m=1}^{k}\e_l\e_m\langle
   \alpha_l,\alpha_m\rangle\right)h.\end{array}\end{equation}
Using this we arrive at the following theorem:
\begin{theorem}\label{summary-D's}
  The function $D(z,s,\underline\alpha)$ admits meromorphic
  continuation to $\C$. Furthermore it satisfies the following:
\begin{enumerate}
\item\label{itscold} The  poles of $D(z,s,\underline\alpha)$ are included in  the singular set
  for the surface $M$, and the pole order at a singular point is at
  most $k$.
\item \label{recursion} For $\Re(s)>1/2$,  and $s$ not in the singular set, the function  $D(z,s,\underline\alpha)$ is square integrable and
  satisfies
$$D(z,s,\underline\alpha)=-R(s)\left(\sum_{l=1}^k\left.\partial_{\e_l}L(\underline
 \e)\right\vert_{\underline\e=\underline
 0}D(z,s,\underline\alpha_l)+\sum_{1\leq m<l\leq k}\left.\partial_{\e_l,\e_m}L(\underline
 \e)\right\vert_{\underline\e=\underline
 0}D(z,s,\underline\alpha_{l,m})\right),$$
where $\underline\alpha_l$ is $\underline\alpha$ with the $l$-th
component removed.  This equation provides the analytic continuation of
$D(z,s,\underline\alpha)$ using the meromorphic continuation of the
Green's function. The analytically continued function grows at most polynomially as
$z$ tends to a cusp.
  
\item \label{itsstillcold}For $1/2<\sigma_0< \Re(s)<\sigma_1$,  and $s$ not in the singular
  set, the function  $D(z,s,\underline \alpha)$ grows at most polynomially as
  $\abs{\Im(s)}\to\infty$, and $z$ is in a compact set.
\item \label{functional-equation} The function  $D(z,s,\underline\alpha)$ satisfies a functional
  equation. This is derived  by multiplying (\ref{functional-equation-E}) by
  $U(-\underline \e)$ and differentiating both sides, using
  (\ref{original-petridis-obs}).
\end{enumerate}
\end{theorem}
\begin{proof}[Proof:] (\ref{itscold}), (\ref{recursion})
  and(\ref{itsstillcold}) can be found in \cite{petridis-risager}, and
  (\ref{functional-equation}) follows from differentiation
  (\ref{functional-equation-E}). See also \cite{risagerthesis}.
\end{proof}
\begin{remark}\label{multivariate-functional-eq}
 An example of the functional equation in Theorem \ref{summary-D's}
 (\ref{functional-equation}) is 
\begin{align*}D(z,s,\alpha_1,\alpha_2)&=\phi(s,\underline
  0)D(z,1-s,\alpha_1,\alpha_2)\\
&\quad +\left.\partial_{\e_1}\phi(s,\underline\e)\right\vert_{\underline \e=\underline 0}D(z,1-s,\alpha_2)+\left.\partial_{\e_2}\phi(s,\underline
\e)\right\vert_{\underline \e=\underline 0}D(z,1-s,\alpha_1)\\ &\quad +\left.\partial_{\e_2,\e_1}\phi(s,\underline
\e)\right\vert_{\underline \e=\underline 0}E(z,1-s).\end{align*}
\end{remark}

\begin{remark} We notice that, although Theorem \ref{summary-D's}
  concerns $D(z,s,\underline \alpha)$, where
  $\alpha=(\alpha_1,\ldots,\alpha_l)$ with $\alpha_l$ compactly
  supported, we can also handle non-compact but cuspidal cohomology in
  the following way: Since Theorem \ref{summary-D's} immediately
  gives --through (\ref{original-petridis-obs}) --  the
  properties of $\left.\frac{\partial^k}{\partial
  \e_1\cdots\partial\e_k}E(z,s,\underline \e)\right\vert_{\underline
\e=\underline 0}$,  which is invariant under shift of $\alpha_l$ within
its cohomology class. Since for  every $f(z)\in S_{2}(\Gamma)$
the harmonic 1-form $\Re(f(z)dz)$ has a compactly supported form in
its cohomology class, we see that Theorem \ref{summary-D's} provides the
analytic properties of the series 
\begin{equation}\label{d-n-mseries}
D^{n_1,\ldots , n_k}(z, s, \omega_1, \ldots , \omega_k)=\sum_{\gamma\in \GinfmodG}\prod_{l=1}^k \left(\int_{i\infty}^{\gamma z}\omega_l\right)^{n_l}
 \Im (\gamma z)^s,
\end{equation}
where $\omega_i$, $i=1, \ldots, k$, are complex or real harmonic
cuspidal 1-forms.

We note also that the  \lq differentiated scattering matrices\rq{}  $$\left.\partial_{\e_k,\ldots,\e_1}\phi(s,\underline
\e)\right\vert_{\underline \e=\underline 0}$$ are  invariant under shift of $\alpha_l$ within
its cohomology class.  We shall freely use these connections below. 
\end{remark}
\begin{remark} \label{odd-implies-zero}It is well known (see e.g. \cite[page 218, Remark 61]{hejhal}) that in the
  one-cusp case the scattering matrix is \emph{even} in the character
  (i.e.  $\phi(s,\chi)=\phi(s,\overline{\chi})$). It follows that $$\left.\partial_{\e_k,\ldots,\e_1}\phi(s,\underline
\e)\right\vert_{\underline \e=\underline 0}=0$$ whenever $k$ is
odd. 
\end{remark}

\subsection{Dissolving and special values of Dirichlet series}
 By Theorem \ref{maintheorem-for-characters} the Phillips-Sarnak
condition for the perturbation induced by $\omega$ is equivalent
to$$\res_{s=s_j}D^1(z, s, \omega)\ne 0.$$

The following lemma identifies situations where the Phillip-Sarnak condition
is not satisfied. This is seen as follows:  
\begin{lemma}\label{poweron}
Let $M$ be a finite volume hyperbolic surface of genus $g$.  Let
$\lambda$ be an eigenvalue $>1/4$ of multiplicity $m$. If $g>2m$, there exists a holomorphic cusp form $f(z)$ of weight $2$ such that 
for both perturbations induced by the harmonic 1-forms $\omega_1=\Re (f(z)\, dz)$ and
$\omega_2 =\Im (f(z)\, dz)$ as in
(\ref{muffinstogo}) the Phillips--Sarnak condition for dissolving the
eigenvalue $\lambda$ is \emph{not} satisfied, i.e.
\begin{equation}\label{2phillips-sarnak}
\res_{s=s_j}D^1(z, s, \omega_i)=0, \quad i=1, 2.\end{equation}
\end{lemma}
Note that the condition (\ref{2phillips-sarnak}) implies that
$\res_{s=s_j}D^1(z, s, \omega)=0$ for all $\omega$ in the linear
complex span
of $\omega_1$, $\omega_2$, in particular for $f(z)dz$. 
\begin{proof} We have $M=\GmodH$ for some discrete cofinite subgroup
  $\G$. Let $(u_l)_{l=1}^m$ be a basis for  the eigenspace corresponding to the
  eigenvalue $\lambda$. 
Consider the linear real map $\Lambda:S_2(\Gamma)\to \R^{4m}$ 
 sending $f\in
S_2(\Gamma)$ to $$\left(\Re\inprod{\res_{s=s_j}D^{(1)}(z,s,\w_i)}{u_l},\Im
\inprod{\res_{s=s_j}D^{(1)}(z,s,\w_i)}{u_l}\right)_{\substack{i=1,2\\l=1,\ldots
  m}}, $$ which is
well-defined by Theorem \ref{summary-D's} (\ref{recursion}).
Since $\dim_\R S_2(\Gamma)=2g$, the dimension formula implies that
$\Lambda$ has non-trivial kernel when $2g-4m>0$
which gives the required result.
\end{proof}
\begin{remark}  There are numerically many known
examples of surfaces of genus $g>2$ that has simple eigenvalues. For
these Theorem \ref{poweron} can be applied.
Moreover the proof of the lemma shows that the real dimension of the
relevant $f$'s is at least $2g-4m$. Note however that the kernel of
$\Lambda$ is in fact a complex space since $\omega_1(if)=-\omega_2(f)$
and $\omega_2(if)=\omega_1(f)$. Hence this space must have complex
dimension at least $g-2m$
\end{remark}
We now introduce a Dirichlet series that plays a major role in investigating
movement of an embedded eigenvalue if the
Phillips-Sarnak condition is not satisfied. Let $f(z)=\sum_{n=1}^{\infty}a_ne^{2\pi i nz}$
be the Fourier expansion of $f(z)$ at the cusp $i\infty$.
Let $$u_j(z)=\sum_{n\ne 0}b_n \sqrt{y}K_{s_j-1/2}(2\pi \abs{n} y)e^{2\pi i nx}$$
be the Fourier expansion of $u_j$, which for simplicity we may assume
to be real-valued. We introduce the antiderivative of $f(z)$ as
$$F(z)=\int_{i\infty}^zf(w)\, dw=\sum_{n=1}^{\infty}\frac{a_n}{2\pi i n}e^{2\pi i nz}.$$
We define the Dirichlet series
\begin{equation}\label{Dirichlet series}
L(u_j\otimes F^2, s)=\sum_{n=1}^{\infty} \left(  \sum_{k_1+k_2=n}\frac{a_{k_1}}{k_1}\frac{a_{k_2}}{k_2}             b_{-n}    \right)\frac{1 }{n^{s-1/2}}.
\end{equation}
Since $a_n, b_n$ grow at most polynomially in $n$ we easily see that the above series converges absolutely for $\Re(s)$ sufficiently large.

By unfolding and inserting the relevant Fourier expansions we have, for $\Re(s)$ sufficiently large,
\begin{equation}\label{eq:unfolding-series}
\begin{split}&\inprod{D^2(z,s,f(z)dz)}{u_j}=\int_0^\infty\int_0^1\left(\int_{i\infty}^z
    f(w)dw\right)^2y^s u_j(z)dx\frac{dy}{y^2}\\
 &\qquad =\int_0^\infty\int_0^1\left(\sum_{n=1}^\infty \frac{a_n}{2\pi i n}e(nz)\right)^2y^s\sum_{n\ne 0}b_n \sqrt{y}K_{s_j-1/2}(2\pi \abs{n} y)e(nx)
 dx\frac{dy}{y^2}\\
&\qquad =\frac{-1}{4\pi^2}\sum_{n=1}^\infty\int_0^\infty \left(\sum_{k_1+k_2=n}\frac{a_{k_1}}{k_1}\frac{a_{k_2}}{k_2}\right)
b_{-n} e^{-2\pi n y}y^{s-1/2}K_{s_j-1/2}(2\pi n y)\frac{dy}{y}\\
&\qquad =\frac{-1}{4\pi^2}\sum_{n=1}^\infty\left(\sum_{k_1+k_2=n}\frac{a_{k_1}}{k_1}\frac{a_{k_2}}{k_2}\right)
b_{-n}\frac{1}{(2\pi n )^{s-1/2}}\int_0^\infty
e^{-t}t^{s-1/2}K_{s_j-1/2}(t)\frac{dt}{t}\\
&\qquad =\frac{- 1}{2^{2s+1}\pi^{s+1}}L(u_j\otimes F^2, s)\frac{\Gamma(s+s_j-1)\Gamma(s-s_j)}{\Gamma(s)},
\end{split}
\end{equation}
where we have used \cite[6.621 3]{gradshtejn-ryzhik}. Using this we can now prove the basic properties of $L(u_j\otimes F^2,
s)$.

\begin{proposition} \label{prop:the-specific-L-f}The series $L(u_j\otimes F^2,s)$ admits
  meromorphic continuation to $s\in \mathbb C$ with possible poles on
  the singular set. The poles are at most of first order.
 Furthermore
   we have the following functional
   equation: Let $$\Lambda(u_j\otimes F^2,s)=\frac{1}{(4\pi)^s}\frac{\Gamma(s+s_j-1)\Gamma(s-s_j)}{\Gamma(s)}L(u_j\otimes F^2,s).$$
Then
\begin{equation*}
  \Lambda(u_j\otimes F^2,s)=\phi(s)\Lambda(u_j\otimes F^2,1-s)
\end{equation*}  
where $\phi(s)$ is the scattering matrix.
\end{proposition}
\begin{proof} This follows from (\ref{eq:unfolding-series}) and the
  properties of $D^2(z,s,f(z)dz)$ as recorded in Theorem \ref{summary-D's}:   Since the left hand side of
  (\ref{eq:unfolding-series}) is meromorphic for $s\in \C$
this
  immediately gives meromorphic continuation of  $L(u_j\otimes F^2,s)$
  to $s\in \C$. Since 
$$D^{2}(z,s,f(z)dz)=D^{2}(z,s,\omega_1)-D^{2}(z,s,\omega_2)+2iD^{1,1}(z,s,\omega_1,\omega_2),$$
we have by Theorem \ref{summary-D's} that
$D^{2}(z,s,f(z)dz)=-R(s)(\psi(z,s))$ where $\psi(z,s)$ has at most a
simple pole on the singular set. But then 
$$\inprod{D^2(z,s,f(z)dz)}{u_j}=\inprod{\psi(z,s)}{R(\overline{s})u_j}=\frac{1}{s(1-s)-\lambda_j}\inprod{\psi(z,s)}{u_j}$$
which then holds for $s\in \C$ by meromorphic continuation. Comparing
with (\ref{eq:unfolding-series})  and noticing that $\Gamma(s-s_j)$
has a simple pole at $s=s_j$ we prove that the poles of $L(u_j\otimes
F^2,s)$ are all at most simple.

 The functional equation in Theorem \ref{summary-D's}  (\ref{functional-equation}),
   reduces, in this case, to $$D^{2}(z,s,f(z)dz)=\phi(s)D^{2}(z,1-s,f(z)dz)+\phi^{(2)}(s,f(z)dz)E(z,1-s).$$
The fact that $\inprod{E(z,s)}{u_j}=0$ and Equation  (\ref{eq:unfolding-series}) give
\begin{gather}\begin{split}\frac{1}{4^{s}\pi^{s}}&L(u_j\otimes
    F^2, s)\frac{\Gamma(s+s_j-1)\Gamma(s-s_j)}{\Gamma(s)}\\
&=\phi(s)[\textrm{same expression evaluated at $1-s$}].\\  
\end{split}\end{gather}

\end{proof}
We note that in the case of multiple cusps the above functional
equation becomes more complicated since in general 
$\phi^{(1)}(s,f(z)dz)$ can be  a non-zero matrix (with diagonal
entries equal to zero). 
\begin{lemma}
Assume  (\ref{2phillips-sarnak}), i.e.  that the Phillips-Sarnak
condition is not satisfied for the  perturbations induced by both $\omega_i$,
$i=1, 2$. Then $L(u_j\otimes F^2,s)$ is regular at $s=s_j$.
\end{lemma}
\begin{proof} As in the proof of Proposition~\ref{prop:the-specific-L-f}
  we have that $D^{2}(z,s,f(z)dz)=-R(s)(\psi(z,s))$ where $\psi(z,s)$
  has at most a simple pole at $s_j$. But assuming
  (\ref{2phillips-sarnak}), it follows easily from Theorem
  \ref{summary-D's} (\ref{recursion}), that $\psi(z,s)$ is in fact
  regular since the only potential poles would come from $D^1(z,s,\w_i)$. Therefore, as in the proof of
  Proposition~\ref{prop:the-specific-L-f}, we conclude that
  $\inprod{D^2(z,s,f(z)dz)}{u_j}$ has at most a simple pole at
  $s_j$. Comparing with (\ref{eq:unfolding-series})  and again using that $\Gamma(s-s_j)$
has a simple pole at $s=s_j$ gives the claim.
\end{proof}

\begin{theorem}\label{dirichlet-series-condition}
Assume  (\ref{2phillips-sarnak}), i.e. that the Phillips-Sarnak
condition is not satisfied under perturbations induced by both $\omega_i$, $i=1, 2$, and   that $L(u_j\otimes F^2, s_j)\ne 0$.
For all directions $\w$ in the  real span of
$\w_1,\w_2$  with at most two exceptions we have $$\Re\hat{s}_j^{(4)}(0,
\w)\ne 0.$$ In particular there exists a cusp form with eigenvalue $s_j(1-s_j)$ that
is dissolved in this direction.

\end{theorem}
\begin{proof} The Phillips-Sarnak condition will not be satisfied in
  the whole span of $\w_1$, $\w_2$. Assume that $\Re\hat{s}_j^{(4)}(0,
\w)= 0$ for three distinct directions given by $\eta_k=a_k\w_1+b_k\w_2$,
$k=1,2,3$, i.e. $(a_kb_l-a_lb_k)\neq 0$ for $k\neq l$. We have
\begin{equation*}
  D^2(z,s,\eta_k)=a_k^2 D^2(z,s,\w_1)+b_k^2D^2(z,s,\w_2)+2a_kb_kD^{1,1}(z,s,\w_1,\w_2)  .
\end{equation*}
We can solve for $D^2(z,s,\w_1)$, $D^2(z,s,\w_2)$, and
$D^{1,1}(z,s,\w_1,\w_2) $ as long as the following determinant is
nonzero:
\begin{equation*}
  \abs{\begin{array}{ccc}a_1^2&b_1^2& 2a_1b_1\\
a_2^2&b_2^2& 2a_2b_2\\
a_3^2&b_3^2& 2a_3b_3
\end{array}}=-2a_1^2a_2^2a^2_3\abs{\begin{array}{ccc}1&b_1/a_1& (b_1/a_1)^2\\
1&b_2/a_2& (b_2/a_2)^2\\
1&b_3/a_3& (b_3/a_3)^2\\
\end{array}}
=2\prod_{k>l}(a_kb_l-a_lb_k)\neq 0,
\end{equation*}
where we have used the fact that the last matrix is Vandermonde. 
Since $\Re\hat{s}_j^{(4)}(0,
\eta_k)= 0$, it follows from Theorem \ref{maintheorem} that
$D^2(z,s,\eta_k)$ is regular at $s_j$, and by solving the above system
that $D^2(z,s,\w_1)$, $D^2(z,s,\w_2)$, and
$D^{1,1}(z,s,\w_1,\w_2)$ are regular. This
implies that 
\begin{equation*}D^{2}(z,s,f(z)dz)=D^{2}(z,s,\omega_1)-D^{2}(z,s,\omega_2)+2iD^{1,1}(z,s,\omega_1,\omega_2)\end{equation*}
is regular also at $s=s_j$. By (\ref{eq:unfolding-series}) it follows that $L(u_j\otimes F^2, s)$ must have a zero at $s_j$, since the Gamma factor
$\Gamma(s-s_j)$ has a pole at that point. This contradicts the
assumption of the theorem.

\end{proof}

\begin{remark}
We note that the Dirichlet series $$D(s)=\sum_{n=1}^\infty \sum_{j=1}^{n-1}\frac{a_{n-j}a_j}{j}\frac{1}{n^s}=\frac{1}{2}\sum_{n=1}^\infty \sum_{j=1}^{n-1}\frac{a_{n-j}a_j}{(n-j)j}\frac{1}{n^{s-1}}$$
was recently studied by Diamantis, Knopp, Mason, O'Sullivan and
Deitmar \cite{deitmardiamantis,diamantisknoppmasonosullivan}. The
series $L(u_j\otimes F^2, s)$ (See \ref{Dirichlet series}) has the structure of a
Rankin-Selberg convolution between $D(s)$ and the $L$-function for $u_j$.  
\end{remark}
\begin{remark}
The special value in Theorem \ref{dirichlet-series-condition} is on
the critical line and at the same height as the trivial zeros of
$L(u_j\otimes F^2, s)$. It is of interest to study the non-vanishing
of these special values theoretically and/or numerically.

\end{remark}

There is a generalization of Theorem \ref{dirichlet-series-condition}
which we describe briefly: Define
\begin{equation}
  \label{dirichlet-series-generalized}
  L(u_j\otimes F^l, s)=\sum_{n=1}^{\infty} \left(
    \sum_{k_1+\cdots+k_l=n}\frac{a_{k_1}}{k_1}\cdots\frac{a_{k_l}}{k_l}
    b_{-n}    \right)\frac{1 }{n^{s-1/2}}.
\end{equation}
By essentially the same computation as in (\ref{eq:unfolding-series})
we have 
\begin{equation}\label{choir}
\inprod{D^n(z,s,f(z)dz)}{u_j}=\frac{- 1}{2^{2s+1}\pi^{s+1}}L(u_j\otimes F^n, s)\frac{\Gamma(s+s_j-1)\Gamma(s-s_j)}{\Gamma(s)}.
\end{equation}
Notice that this computation also proves the meromorphic continuation
of $L(u_j\otimes F^n, s)$ to $s\in \C$, and that at $s_j$ the function
$L(u_j\otimes F^n, s)$ has a pole of order at most $n-1$.

\begin{theorem}
 Assume that $L(u_j\otimes F^n, s)$ does not have a zero at $s_j$.  For all
 directions $\w$ in the real span of $\w_1$, $\w_2$ with at most $n$ exceptions we have
$$\Re \hat{s}_j^{(2r)}(0, \omega)\ne 0,$$ for some $r\leq n$. In particular there exists a cusp form with eigenvalue $s_j(1-s_j)$ that is
dissolved in this direction.

\end{theorem}
\begin{proof} Assume that $\Re\hat{s}_j^{(2r)}(0,
\eta_k)= 0$, $r=1,\ldots, n$ for $n+1$ distinct directions given by $\eta_k=a_k\w_1+b_k\w_2$,
$k=0,\ldots n$, i.e. $(a_kb_l-a_lb_k)\neq 0$ for $k\neq l$. We have
\begin{equation*}
  D^n(z,s,\eta_k)=\sum_{l=0}^n\binom{n}{l}a_k^lb_k^{n-l}D^{l,{n-l}}(z,s,\w_1,\w_2)
\end{equation*}
We can solve for all $D^{l,{n-l}}(z,s,\w_1,\w_2)$ as long as the following determinant is
nonzero:
\begin{align*}
  \abs{\binom{n}{l}a_k^{n-l}b_k^l
  }_{k,l=0}^n&=\left(\prod_{l=0}^n\binom{n}{l}\right)\abs{a_k^{n-l}b_k^l
  }_{k,l=0}^n=\left(\prod_{l=0}^n\binom{n}{l}a_l^n\right)\abs{(b_k/a_k)^l
  }_{k,l=0}^n\\ &=\left(\prod_{l=0}^n\binom{n}{l}\right)\prod_{k<l}(a_kb_l-a_lb_k)\neq 0,
\end{align*}
where again we have used the fact that the last matrix is
Vandermonde. It follows that all $D^{n-{l},l}(z,s)$
are linear combinations of $D^{n}(z,s,\eta_k)$, $k=0,\ldots n.$

Since $$\Re \hat{s}_j^{(2r)}(0, \eta_k)= 0,$$
Theorem~\ref{maintheorem-for-characters} allows us to conclude that $D^{n}(z,s,\eta_k)$ is regular at $s_j$, and
therefore also that $D^{n-{l},l}(z,s)$ is regular at $s_j$. Since 
\begin{equation*}
  D^n(z,s,f(z)dz)=\sum_{j=0}^n\binom{n}{j}i^{n-j}D^{j,n-j}(z,s,\w_1,\w_2)
\end{equation*}
it follows also that $ D^n(z,s,f(z)dz)$ is regular at $s_0$.

Therefore the left of (\ref{choir}) is regular. It follows that the
right-hand side of (\ref{choir}) is regular, which proves -- since
$\Gamma(s-s_j)$ has a pole at $s_j$ --  that
$L(u_j\otimes F^n, s)$ has a zero at $s_j$.  But this contradicts the assumption of the theorem. 
\end{proof}

\section{Are Maa{\ss} cusp forms isolated in deformation space?}\label{subvarieties-section}
In this section we investigate the question of whether cuspidal
eigenvalues are isolated in deformation spaces.
As before, we assume that $M$ has one cusp only.

Let $\omega_1,\ldots, \omega_{2g}$ be a basis for the space
$\hbox{Harm}_\R^1(M)$ of real harmonic
cuspidal 1-forms on $M$ , and let $\alpha_i$ be compact differentials in the same
cohomology classes. We consider the corresponding multiparameter perturbations as in
section \ref{multivariate}.

Since we know that $\Re(s_j(\underline{\e}))$ has a maximum at $\underline{\e}=\underline{0}$, we know that its Hessian must be negative semi-definite. If  the Hessian were negative definite, the
maximum would be strict, i.e. the cuspidal eigenvalue would be
isolated. We  show below  that in most cases the Hessian is not strictly
negative definite, so we do not  know  apriori that the eigenvalue should be isolated.

We start by finding an expression for the Hessian:

 \begin{proposition} \label{general-hessian}Let $s_j=s_j(0)\neq 1/2$ be an embedded eigenvalue 
  of multiplicity $m$. Then the Hessian matrix $H=(h_{kl})_{k, l=1, \ldots, 2g}$ of $\hat s_j(\underline{\e})$ at
  $\underline{\e}=\underline{0}$ is given by
$$h_{kl}=-\frac{1}{m}\Re\inprod{\res_{s=s_j}D(z,s,\alpha_k)}{\res_{s=s_j}D(z,s,\alpha_l)}.$$
\end{proposition}
\begin{proof} The diagonal terms $h_{kk}$ have already been computed in Theorem~\ref{maintheorem-for-characters} (2), and the
  off-diagonal terms can be found similarly:
we note that by  (\ref{Dn-residue}) we have
\begin{equation}\label{Dn-residue-special}\res_{s=s_j}D(z,s,\alpha_k)=\frac{1}{2s_j-1}\sum_{i=1}^m\inprod{\partial_{\e_k}L(\underline{\e})\vert_{\underline\e=\underline 0} E(z,s)}{u_{j,i}}u_{j,i}.\end{equation}
We repeat Eq. (\ref{flocke})
\begin{equation}\label{williteverend}2m\Re (\hat{s}(\underline{\e})-s_j)=-\frac{1}{2\pi
  i}\int_{\partial B(s_j,u)}(s-s_j)\frac{\phi'(s, \underline\e)}{\phi(s, \underline\e)}\, ds.\end{equation}
An analysis similar to the one in the  proof of Theorem~\ref{maintheorem}
gives
\begin{equation}\label{lookatthis}\partial_{\e_k,\e_l}2m\Re
(\hat{s}(\underline{\e})-s_j)\vert_{\underline{\e}=\underline{0}}=\frac{\res_{s=s_j}\partial_{\e_k,\e_l}
\phi(s,\underline{\e})\vert_{\underline\e=\underline 0}}{\phi(s_j,0)}.\end{equation}
To compute $\res_{s=s_j}\partial_{\e_k,\e_l}
\phi(s,\underline{\e})\vert_{\underline\e=\underline 0}$
we use Remark \ref{multivariate-functional-eq} and the same technique
as in the proof of Proposition \ref{scatteringphin}.
and find that $(2s-1)\partial_{\e_k,\e_l}\phi(s,\underline{\e})\vert_{\underline\e=\underline
  0}$ equals
$$
\int_{M}E(z,s)\left(\partial_{\e_l}
L(\underline{\e})\vert_{\underline\e=\underline 0} D(z,s,\alpha_k)+\partial_{\e_k}
L(\underline{\e})\vert_{\underline\e=\underline 0} D(z,s,\alpha_l)+ \partial_{\e_k,\e_l}
L(\underline{\e})\vert_{\underline\e=\underline 0} D(z,s) \right)d\mu(z).$$
We note that the term involving $\partial_{\e_k,\e_l}
L(\underline{\e})\vert_{\underline\e=\underline 0} D(z,s)$ is regular
so  it does not contribute to the residue in (\ref{lookatthis}). The rest equals
$$
\int_{M}D(z,s,\alpha_l)\partial_{\e_k}L(\underline{\e})\vert_{\underline\e=\underline 0} E(z,s) + D(z,s,\alpha_k)\partial_{\e_l}L(\underline{\e})\vert_{\underline\e=\underline 0} E(z,s) 
d\mu(z).$$
Taking residues and using (\ref{get-to-conjugate})  we find that $\res_{s=s_j}\partial_{\e_k,\e_l}
\phi(s,\underline{\e})\vert_{\underline\e=\underline 0}$ equals
$$\frac{-\phi(s_j)}{2s_j-1} \left(\inprod{\partial_{\e_k}L(\underline{\e})\vert_{\underline\e=\underline 0} E(z,s)}{\res_{s=s_j}D(z,s,\alpha_l)} + \inprod{\partial_{\e_l}L(\underline{\e})\vert_{\underline\e=\underline 0} E(z,s)}{\res_{s=s_j}D(z,s,\alpha_k)} \right).
$$
We insert (\ref{Dn-residue-special}). This makes the above expression equal to
\begin{eqnarray*}\frac{-\phi(s_j)}{\abs{2s_j-1}^2}&\displaystyle \Big(\sum_{i=1}^m \inprod{\partial_{\e_k}L(\underline{\e})\vert_{\underline\e=\underline 0} E(z,s)}{u_{j,i}}\inprod{u_{j,i}}{\partial_{\e_l}L(\underline{\e})\vert_{\underline\e=\underline 0} E(z,s)}\\&+  \inprod{\partial_{\e_l}L(\underline{\e})\vert_{\underline\e=\underline 0} E(z,s)}{u_{j,i}}\inprod{u_{j,i}}{\partial_{\e_k}L(\underline{\e})\vert_{\underline\e=\underline 0} E(z,s)}\Big)\\
=-\phi(s_j)2\Re\frac{1}{\abs{2s_j-1}^2}&\displaystyle\sum_{i=1}^m \inprod{\partial_{\e_k}L(\underline{\e})\vert_{\underline\e=\underline 0} E(z,s)}{u_{j,i}}\inprod{u_{j,i}}{\partial_{\e_l}L(\underline{\e})\vert_{\underline\e=\underline 0} E(z,s)},
\end{eqnarray*}
which gives the desired result.

\end{proof}

\begin{remark} We note that since the Hessian is $-m^{-1}$ times the real
  part of a Gram matrix we see that it is indeed negative
  semi-definite, as is should be.
\end{remark}
\begin{remark} It is not hard to see that the Hessian  
in many cases fails to be strictly negative definite. For $g>m$,   where $m$ is the
  multiplicity of the  cuspidal eigenvalue, the null space of the
  Hessian is at least of dimension $2(g-m)$. This is seen as follows: Consider
  the map between real vector spaces $M:\hbox{Harm}_\R^1(M)\to \R^{2m}$ defined
  by $$\omega\mapsto \left(\Re\inprod{\res_{s=s_j}D(z,s,\w)}{u_l},\Im
\inprod{\res_{s=s_j}D(z,s,\w)}{u_l}\right)_{l=1,\ldots
  m}, $$
which must have a kernel of dimension $\geq 2g-2m$. We note that for
elements $\omega$ in this kernel we have $\res_{s=s_j}D(z,s,\omega)=0$. By choosing a
basis for this kernel and extending it to a basis for
$\hbox{Harm}_\R^1(M)$ we see from the expression in Proposition \ref{general-hessian}
that in this basis the Hessian is expressed by a matrix which has all
zeroes outside a block of size $2m\times2m$.  
\end{remark}

\subsection{Symmetries} 
We now want to consider surfaces with certain symmetries.

Define $T(z)=-\overline{z}$. Clearly $T^2=I$.
An (eigen)function $h$ is  even (has parity 1) resp. odd (has parity
-1) if $h(Tz)=\pm h(z)$.
A one form $\alpha=p(x, y)\, dx+q(x, y)\, dy$ is even resp. odd if $\alpha\circ T=\pm \alpha$.
This means that $p$ is odd resp. even, while $q$ is even resp. odd. If
$\alpha=\Re(f(z)dz)$ then $\alpha$ is even resp. odd if $f$ has
strictly imaginary resp. real Fourier coefficients.
A Fuchsian group $\Gamma$ has fundamental domain symmetric with
respect to the imaginary axis if $\Gamma=T\Gamma T$. In the rest of
this section we assume this to be case. 
 
With these assumptions we easily see that  the Eisenstein series is
even for all $s$.
By (\ref{firstLvariation}) we see that, if $\alpha$ is even resp. odd, then $L^{(1)}(0)$ 
preserves the parity of $h$ resp. changes the parity of $h$.

The Phillips--Sarnak integral $\langle L^{(1)}(0) u_j, E(z,
s_j)\rangle$ clearly vanishes if  the inner product
 $\langle L^{(1)}(0) u_j, E(z, s)\rangle=0$ for \emph{all} $s$. This happens
 automatically if $\alpha$ is even and $u_j$ is odd, or $\alpha$ is
 odd and $u_j$ is even.

\begin{lemma}\label{parity-Ds}Let the $1$-forms $\omega_l$ have parity $\eta_l\in\{1, -1\}$ for $l=1, \ldots k$. Then the parity of 
$D^{n_1, \ldots , n_k}(z, s, \omega_1, \ldots, \omega_k)$ in (\ref{d-n-mseries})  is $\prod_{l=1}^k \eta_l^{n_l}$.
\end{lemma}
\begin{proof} For $\gamma\in \Gamma$ we write $\gamma'=T\gamma T$, which varies over all of $\Gamma$. Since $i\infty$ is fixed by $T$, we have that $ \Gamma_{\!\!\infty} =T\Gamma_{\!\!\infty}T$. We have
$$D(Tz, s, \omega_1, \ldots, \omega_k)=\sum_{\gamma\in \GinfmodG}\prod_{l=1}^k\left(\int^{\gamma Tz}_{i\infty}\omega_l\right)^{n_l}\Im (\gamma Tz)^s$$
$$   =\sum_{\gamma'\in \GinfmodG}\prod_{l=1}^k\left(\int^{T\gamma' z}_{Ti\infty}\omega_l\right)^{n_l}\Im (T\gamma 'z)^s=\sum_{\gamma'\in \GinfmodG}\prod_{l=1}^k\left(\int^{\gamma' z}_{i\infty}\omega_l\circ T\right)^{n_l}\Im (\gamma 'z)^s$$
$$=\prod_{l=1}^k\eta_l^{n_l}\sum_{\gamma'\in \GinfmodG}\prod_{l=1}^k\left(\int^{\gamma' z}_{i\infty}\omega_l\right)^{n_l}\Im (\gamma 'z)^s=\prod_{l=1}^k\eta_l^{n_l}D(z, s, \omega_1, \ldots, \omega_k).$$
\end{proof}
Recall that for surfaces with one cusp the scattering function is
given by (compare (\ref{karakter}))
\begin{equation}\label{phiseries}\phi(s, \underline{\e})=\sqrt{\pi}\frac{\Gamma (s-1/2)}{\Gamma (s)}\sum_{\gamma\in \Gamma_{\infty}\setminus \Gamma /\Gamma_{\infty}}\frac{\chi(\gamma, \underline{\e})}{|c|^{2s}}
\end{equation}
when  $\Re(s)>1$. We assume that all $\alpha_i$ are eigenfunctions of $T$ with parity $\eta_i$.
\begin{lemma}\label{parity-phi} Let $\alpha_i$ have parity $\eta_i\in \{ +1, -1\}$. The scattering function satisfies $$\phi(s,
  (\eta_i\e_i)_{i=1}^n)=\phi(s, (\e_i)_{i=1}^n).$$
\end{lemma}
\begin{proof}
We have $c(T\gamma T)=-c(\gamma)$ and $\int_{z_{0}}^{T\gamma
  Tz_0}\alpha_i=\eta_i\int_{Tz_{0}}^{\gamma Tz_0}\alpha_i$. Now the result follows easily using (\ref{phiseries}).
\end{proof}
\begin{corollary}\label{lots-zero} Let $\displaystyle \sum_{\substack{i=1\\\alpha_i
     \rm{  odd}}}^{n}n_i$ or $\displaystyle \sum_{i=1}^{n}n_i$ be odd. Then 
$$\partial_{\e_1}^{n_1}\cdots\partial_{\e_{2g}}^{n_{n}}\phi(s,\e)\vert_{\underline{\e}=\underline 0}=0.$$
\end{corollary}
\begin{proof} If $\displaystyle \sum_{\substack{i=1\\\alpha_i
      \rm{odd}}}^{n}n_i$ is odd differentiate the equality in Lemma \ref{parity-phi} and
  plug $\underline \e=\underline 0$. If $\displaystyle
  \sum_{i=1}^{n}n_i$ is odd, use instead Remark \ref{odd-implies-zero}.
\end{proof}
We choose $\alpha_i$ as follows: Let  $f_1,\ldots f_g$ be
a basis for $S_2(\Gamma)$, where $f_j$ has real Fourier
coefficients. (If $\Gamma$ is congruence there is even a basis with
integer coefficients).
%
%
 Then  we let $$\alpha_j=\Re(f_j(z)dz)\textrm{ and
} \alpha_{j+g}=\Re(if_j(z)dz),\quad  j=1,\ldots, g,$$ so  that $\alpha_j$
is odd when $j=1,\ldots,g$, and even when $j=g+1,\ldots,2g$.
With these choices the Hessian has the following form 
for $u_j$ even (an analogous result holds for odd ones):
\begin{proposition}\label{lots-of-zeroes} Assume that  $s_j=s_j(0)$ correspond to a simple even
  embedded eigenvalue. Then outside the lower-right $g\times g$ corner of the Hessian of $s_j(\underline{\e})$ at
  $\underline{\e}=0$  all entries are zero. 
\end{proposition}
\begin{proof}
Outside the lower-right $g\times g$ corner of the Hessian
  we see from Proposition \ref{general-hessian} that
  all entries involve an inner product with one entry being $\res_{s=s_j}D(z,s,\a_l)$
  where $\alpha_l$ is odd. It follows from  Lemma \ref{parity-Ds} that
   $D(z,s,\a_l)$ is odd for such an $\alpha_l$. Its residue at $s_j$ is, therefore, also odd, while it is a multiple of $u_j$, which is even. As a result $D(z,s,\a_l)$ must be regular at $s_j$.  We conclude that
  $\res_{s=s_j}D(z,s,\a_l)=0$ and that the relevant inner product is
  zero.


\end{proof}

%




\begin{remark}\label{terms-zero}Note that many derivatives - e.g. all odd ones - of  $\Re(s_j)(\underline\e)$
  are zero when evaluated at $\underline \e=\underline 0$. This
  follows from combining (\ref{williteverend}) and Corollary
  \ref{lots-zero}. Hence in the  Taylor expansion of
  $\Re(s_j)(\underline\e)$ there are many terms which are automatically zero.
\end{remark}

We now restrict ourselves for simplicity to the case genus $g=1$. Choosing $f(z)$ to
be the unique cusp form with real coefficients and first non-zero
Fourier coefficient equal to 1 we
have that $\alpha_1=\Re(f(z)dz)$  is odd and
$\alpha_2=\Im(f(z)dz)=\Re(if(z)dz)$ is even.

Subgroups of $SL_2(\Z)$ with 1 cusp at
infinity are called cycloidal groups. Examples of cycloidal groups with genus 1 can be found in \cite{stromberg-thesis}. 
We explain another classical construction of a group with the desired properties:

\begin{example} Let $p$ be a prime. Consider $\Gamma_0(p)$
which has 2 cusps. Let $$W_p:=\left(\begin{array}{cc}0&-1/\sqrt{p}\\
    \sqrt{p}&0\end{array}\right)$$ be the Fricke involution. The
corresponding Fricke group $\Gamma_*(p):=\Gamma_0(p)\cup
\Gamma_0(p)W_p$ is a Fuchsian group with finite covolume. It has only 1 cusp as it identifies the
two cusps of $\Gamma_0(p)$. It is clear that an $f\in
S_2(\Gamma_0(p))$ lifts to $S_2(\Gamma_*(p))$ if and only if
$f\vert_2W_p(z)=f(z)$, i.e. if $f$ is an eigenfunction of the Fricke
involution with eigenvalue $+1$. The smallest prime for which
$S_2(\Gamma_*(p))\neq \emptyset$ is $p=37$, in which case it is
one-dimensional, i.e. the genus  is 1. A non-zero element of  $S_2(\Gamma_*(37))$ can be
constructed as follows, see \cite[\textsection 5]{mazur-swinnerton-dyer},
\cite[p. 57 ff.]{elkies} for additional details:
Let $\theta_B$ and $\theta_C$ be the theta functions associated to the
positive definite quadratic forms with matrices
$$B=\left(\begin{array}{cccc}2&1&0&1\\1&8&1&-3\\0&1&10&2\\1&-3&2&12\end{array}\right),\quad C=\left(\begin{array}{cccc}4&1&2&1\\1&4&1&0\\2&1&6&-2\\1&0&-2&20\end{array}\right). $$
Then $\phi=\frac{1}{2}(\theta_B-\theta_C)$ is a newform of $S_2(\G_0(37))$
with $\phi\vert_2W_{37}(z)=\phi(z)$, i.e. $\phi\in
S_2(\Gamma_*(37))$. It has Fourier expansion at infinity
$$\phi(z)=q-2q^2-3q^3+2q^4-2q^5+6q^6-q^7+6q^9+4q^{10}+\cdots,\quad q=e^{2\pi i z}.$$
The $L$-function of $\phi$ equals the zeta-function of the elliptic
curve $y^2+y=x^3-x$, which is 37-A1 in Cremona's table.
\end{example}
Consider an embedded even eigenvalue $s_j$ for the Laplacian for
$\Gamma_*(37)$. We assume for simplicity that the dimension of the
eigenspace is 1. In this case Proposition \ref{lots-of-zeroes} says
that the Hessian is given by 
$$\left(
\begin{array}{cc}
0&0\\
0&-\norm{\res_{s=s_j} D(z,s,\alpha_2)}^2
\end{array}
\right)$$
where $\alpha_2=\Im(\phi(z)dz)$. 

By parity considerations
$\res_{s=s_j} D(z,s,\alpha_1)=0$, where $\alpha_1=\Re(\phi(z)dz)$. It
follows that $\norm{\res_{s=s_j} D(z,s,\alpha_2)}\neq 0$ if and only
if $\res_{s=s_j} D(z,s,\phi(z)dz)\neq 0$. But this happens precisely if
the Rankin-Selberg $L$-function between the corresponding eigenfunction
$u_j$ and the weight 2 cusp form $\phi$ does not vanish at the central point $s_j+1/2$,
i.e. if $$L(u\otimes\phi,s_j+1/2)\neq 0.$$ Hence the lower right
corner of the Hessian -- and therefore  $\frac{d^2}{d\e_2^2}\Re(s_j(\e_1,\e_2))\vert_{\e=0}$--  is non-zero if and only if  $L(u\otimes\phi,s_j+1/2)\neq 0$.
 
\begin{remark} \label{luo}
The value   $L(u\otimes\phi,s_j+1/2)$ is expected to be non-zero for many even cusp forms $u_j$. In
fact Luo proved \cite{luo} that the central
value of the Rankin-Selberg $L$-function of a weight 4 cusp form with a
Maa\ss{} eigenfunction is non-zero for a positive proportion of the eigenfunctions.
\end{remark}

In this setup  we search for  conditions that ensure that the cuspidal
eigenvalue dissolves in a (punctured) neighborhood of the
deformation space  $\hbox{Harm}_\R^1(M)$.

\begin{theorem}\label{destruction-in-a-neighborhood} Assume that
  $\frac{d^2}{d\e_2^2}\Re(s_j(\e_1,\e_2))\vert_{\underline{\e}=\underline 0}\neq 0$ and that
  for some $\m$
$\frac{d^\m}{d\e_1^\m}\Re(s_j(\e_1,\e_2))\vert_{\underline{\e}=\underline 0}\neq 0$. Then 
$$\Re(s_j(\e_1,\e_2))<1/2$$ in a punctured neighborhood of $(0,0)$,
i.e. the cuspidal eigenvalue becomes a resonance in this punctured
neighborhood.
\end{theorem}
\begin{remark} While the assumption
  $\frac{d^2}{d\e_2^2}\Re(s_j(\e_1,\e_2))\vert_{\underline{\e}=\underline 0}\neq 0$ is
  the (standard) Fermi's Golden rule and is related to Rankin--Selberg convolutions, see Remark \ref{luo}, the conditions $\frac{d^\m}{d\e_1^\m}\Re(s_j(\e_1,\e_2))\vert_{\underline{\e}=\underline 0}$ are 
given exactly by the higher order Fermi's golden rules, see Theorem \ref{higher-fermi-intro}. \end{remark}

\begin{proof} This proof is motivated  by
  \cite[p. 366--367]{avelin}. Choose $\m$ minimal such that
  $\frac{d^\m}{d\e_1^\m}\Re(s_j(\e_1,\e_2))\vert_{\underline{\e}=0}\neq 0$. Since
  $\Re(s_j(\e_1,\e_2))\leq 1/2$ we have that $\m$ must be even, say
  equal to $2\m_0$. 

Consider the Taylor expansion 
\begin{equation}\label{taylor-expansion}\Re(s_j(\e_1,\e_2)-1/2)=\sum_{k=2}^\infty\sum_{\substack{n_1+n_2=k\\n_i\geq
0}}c_{n_1,n_2}\e_1^{n_1}\e_2^{n_2}.\end{equation} By
arguing as in Remark \ref{terms-zero} we see that $c_{n_1,n_2}$ is
zero if $n_1+n_2$ is odd or if $n_1$ is odd. Hence it can only be
non-zero if both $n_1$ and $n_2$ are even.

We split $\R^2$ in 3 disjoint sets in the following way. Let
\begin{eqnarray*}
  A_1&=&\{(\e_1,\e_2)\vert \e_1=0\textrm{ or }\e_2= 0\} \\
  A_2&=& \{(\e_1,\e_2)\vert b_1\abs{\e_1}^{\m_0}<\abs{\e_2}<b_2\}\backslash
  A_1\\
  A_3&=& \R^2\backslash (A_1\cup A_2)
\end{eqnarray*}
Here $b_1$ is some appropriate large  positive number and $b_2$ is some
appropriate small positive number to be determined below. We need to  show that for
$0\neq (\e_1,\e_1)\in A_i$ with sufficiently small norm
$\Re(s_j(\e_1,\e_2)-1/2)<0$.

Assume first $0\neq (\e_1,\e_2)\in A_1$. Assume without loss of
generality $\e_2=0$. Then by considering the Taylor expansion of
$\Re(s_j(\e_1,0)-1/2)$ we see using
$\frac{d^\m}{d\e_1^\m}\Re(s_j(\e_1,\e_2))\vert_{\underline \e=\underline 0}< 0$ that
$\Re(s_j(\e_1,\e_2)-1/2)<0$ for $\e_1$ sufficiently small.

Assume next that $0\neq (\e_1,\e_2)\in A_2$. Then $\abs{\e_2}\leq b_2$
and $\abs{\e_1}\leq (b_2/b_1)^{1/\m_0}$. Consider a term with
$n_1+n_2\leq \m$ in the expansion (\ref{taylor-expansion}). If
$(n_1,n_2)\neq (0,2)$ we show that for appropriate choices of $b_1,b_2$
we have 
$$\abs{c_{n_1,n_2}\e_1^{n_1}\e_2^{n_2}}< \frac{1}{2\m^2}\abs{c_{0,2}}\e_2^2$$
If $n_2\geq 2$ this is clear. If $n_2=1$, $c_{n_1,n_2}=0$. If $n_2=0$ the only non-zero term is $c_{\m,0}\e_1^{\m}$ which
can be bounded as
$$\abs{c_{\m,0}\e_1^{\m}}\leq \abs{c_{\m,0}}\abs{\e_2}^2/b_1^2,$$
which makes the claim clear also in this case.
From  
$$\Re(s_j(\e_1,\e_2)-1/2)=\sum_{k=2}^{\m}\sum_{\substack{n_1+n_2=k\\n_i\geq
0}}c_{n_1,n_2}\e_1^{n_1}\e_2^{n_2}+O(\norm{(\e_1,\e_2)}^{\m+1})$$
it follows that for sufficiently small $(\e_1,\e_2)\in A_2$ we have,
using $c_{0,2}<0$, that
$$\Re(s_j(\e_1,\e_2)-1/2)<
\frac{c_{0,2}}{2}\e_2^2+C\abs{\e_2}^{\m+1}< 0,$$
which proves the claim when $(\e_1,\e_2)\in A_2$.

 Lastly we consider $(\e_1,\e_2)\in A_3$ and assume
 $\norm{(\e_1,\e_2)}<b_2$. 
 We now set
$\e_2=u\e_1^{\m_0}$, and we have $\abs{u}\leq b_1$. Then 
$$\Re(s_j(\e_1,u\e_1^{l_0})-1/2)=\sum_{k=2}^\infty\sum_{\substack{n_1+n_2=k\\n_i\geq
0}}c_{n_1,n_2}u^{n_2}\e_1^{n_1+\m_0n_2}.$$
Consider now the terms in this expansion with $n_1+\m_0n_2\leq
2\m_0$. We have $n_2=0,1,2$. If $n_2=0$ then $n_1\leq 2\m_0$. But by
choice of $\m$ among these $c_{n_1,0}$ is only non-zero when
$n_1=2\m_0$. If $n_2=1$ we have $c_{n_1,n_2}=0$.
If $n_2=2$ then we must have $n_1=0$. It follows that 
$$\Re(s_j(\e_1,u\e_1^{l_0})-1/2)=(c_{0,2}u^2+c_{2\m_0,0}))\e_1^{2\m_0}+O(\abs{\e_1}^{2\m_0+1}).$$ 
We note that by our assumptions $(c_{0,2}u^2+c_{2\m_0,0})$ is negative
and bounded away from zero.
Hence  for $\e_1$ small enough 
$$\Re(s_j(\e_1,u\e_1^{l_0})-1/2)<(c_{0,2}u^2+c_{2\m_0,0}))\e_1^{2\m_0}+C(\abs{\e_1}^{2\m_0+1})<0,$$
which finishes the proof. Here $C$ is an absolute constant, since $u$
is bounded.
\end{proof}

\begin{remark}If a cusp form remains on a real analytic subvariety of the
deformation space, as suggested in the
Teichm\"uller case by Farmer and Lemurell \cite{farmerlemurel}, then all the conditions
$\frac{d^\m}{d\e_1^\m}\Re(s_j(\e_1,\e_2))\vert_{\e=0}=0$. But in this case we automatically see that the line
spanned by $\alpha_1$ is contained in such a subvariety. This determines the subvariety, because $\Re (s_j(\e))$ is real analytic. \end{remark}

\begin{remark}In the light of Theorem \ref{destruction-in-a-neighborhood}  it would be very
interesting to investigate numerically $\frac{d^4}{d\e_1^4}\Re(s_j(\e_1,\e_2))\vert_{\underline \e=\underline 0}$, or
equivalently $\res_{s=s_j} D^2(z,s,\alpha_2)$, as the non-vanishing of these would imply
that the cuspidal eigenvalue becomes a resonance in a punctured
neighborhood. So the cuspidal eigenvalue would be isolated in the
deformation space.

\end{remark}

\end{document}